%% file: main-conduche.tex
\documentclass[a4paper,reqno]{amsart}
\input{preamble.tex}

\begin{document}
	\title{On the straightening of every functor}
	\author{Thomas Blom}
    \address{Max Planck Institute for Mathematics, Vivatsgasse 7, 53111 Bonn, Germany}
	\email{blom@mpim-bonn.mpg.de}
	
	\begin{abstract}
		We show that any functor between $\infty$-categories can be straightened. More precisely, we show that for any $\infty$-category $\cC$, there is an equivalence between the $\infty$-category $(\Cat_{\infty})_{/\cC}$ of $\infty$-categories over $\cC$ and the $\infty$-category of unital lax functors from $\cC$ to the double $\infty$-category $\DProf$ of correspondences. The proof relies on a certain universal property of the Morita category which is of independent interest.
	\end{abstract}

	\maketitle 
	\tableofcontents

	\section{Introduction}
	
	One of the cornerstones of higher category theory is Lurie's celebrated straightening-unstraightening theorem for (co)cartesian fibrations \cite[\S 3.2]{HTT}. In the case of cartesian fibrations, this result states that for an $\infty$-category $\cC$, the $\infty$-category $\Cart(\cC)$ of cartesian fibrations over $\cC$ is equivalent to the $\infty$-category of functors from $\cC^\op$ to the $\infty$-category $\Cat_\infty$ of $\infty$-categories. Similar results have been obtained for other types of fibrations. Notably, in \cite[\S 3]{Lurie2009InftyCategoriesGoodwillie} it is shown that the $\infty$-category of \emph{locally} cartesian fibrations over $\cC$ is equivalent to the $\infty$-category of \emph{(unital) lax} functors from $\cC^\op$ to $\Cat_\infty$, and Ayala--Francis showed in \cite[\S 2.3]{AyalaFrancis2020FibrationsInftyCategories} that there is an equivalence between the space of Conduché fibrations\footnote{What we call Conduché fibrations are called exponentiable fibrations in \cite{AyalaFrancis2020FibrationsInftyCategories}} over $\cC$ and the space of functors from $\cC$ to a certain \emph{flagged} $\infty$-category of correspondences.
	
	This may lead one to wonder whether an analogous statement can be proved for \emph{locally} Conduché fibrations---namely, whether these can be straightened to lax functors into an appropriate $\infty$-category of correspondences. Recall that a functor $\cD \to \cC$ is called a \emph{locally cartesian fibration} if for any functor $[1] \to \cC$, the pullback $[1] \times_{\cC} \cD \to [1]$ is a cartesian fibration. By replacing cartesian fibrations with Conduché fibrations in this definition, one arrives at the definition of a \emph{locally Conduché fibration}. However, this notion turns out to be vacuous: since functors to $[1]$ are always Conduché fibrations, it follows that any functor between $\infty$-categories is a locally Conduché fibration. 
	
	Nevertheless, the expected straightening equivalence \emph{does} hold. More precisely, we show the following.
	
	\begin{theoremA}\label{theoremA}
		There exists a double $\infty$-category $\DCorr$ such that for any $\infty$-category $\cC$, there is an equivalence between the $\infty$-category $(\Cat_\infty)_{/\cC}$ of $\infty$-categories over $\cC$ and the $\infty$-category $\UnitLax(\cC,\DCorr)$ of unital lax functors from $\cC$ to $\DCorr$. Moreover, this equivalence is natural in $\cC$.
	\end{theoremA}

	Let us elaborate on what this means. Double $\infty$-categories are defined as Segal objects in $\Cat_\infty$ and can be thought of as categories internal to $\Cat_\infty$. In a double $\infty$-category, one can speak of morphisms between morphisms. This allows one to define functors into a double $\infty$-category that only preserve composition up to a (possibly) non-invertible morphism: instead of asking for an equivalence $F(f) \circ F(g) \simeq F(f \circ g)$, one only asks that there is a morphism from $F(f) \circ F(g)$ to $F(fg)$. Such a functor will be called unital lax. For details on double $\infty$-categories, the reader is referred to \cref{section:double-categories}. In particular, a precise (and coherent) definition of unital lax functors is given in \cref{definition:lax-functors} below. For the naturality statement in \cref{theoremA}, observe that $(\Cat_\infty)_{/\cC}$ is functorial in $\cC$ via pullback and $\UnitLax(\cC, \DCorr)$ via precomposition.
	
	\begin{remark}
		In ordinary category theory, an analog of \cref{theoremA} was proved by Bénabou and written up by Street \cite{Street2001Powerful}. It is worth noting that Bénabou's result, as stated in \cite{Street2001Powerful}, is slightly different, since it uses a bicategory of correspondences. This formulation of the result violates the principle of equivalence, which can be worked around by using a (pseudo) double category of correspondences instead. This is one of the reasons why \cref{theoremA} features a double $\infty$-category of correspondences and not an $(\infty,2)$-category.
	\end{remark}   
	
	A crucial role in our proof of \cref{theoremA} is played by the following universal property of the Morita (double) $\infty$-category, which we believe to be of independent interest.
	
	\begin{theoremA}\label{theoremB:Morita}
		Let $\DD$ be a double $\infty$-category and suppose $\DD$ is \emph{Moritable} (see \cref{def:Moritable}). Then there exists a double $\infty$-category $\Mor(\DD)$, its \emph{Morita double $\infty$-category}, satisfying the universal property that for any double $\infty$-category $\DE$, there is a natural equivalence
		\[\Lax(\DE,\DD) \simeq \UnitLax(\DE,\Mor(\DD))\]
		of categories.
	\end{theoremA}
	
	Intuitively, one can think of $\Mor(\DD)$ as the double $\infty$-category of ``algebras'' in $\DD$ and ``bimodules'' between them, where composition is given by the relative tensor product. In particular, when applying \cref{theoremB:Morita} to a monoidal $\infty$-category $\cC$ (viewed as a double $\infty$-category via \cref{example:Moniodal-as-double}), one recovers its usual Morita double $\infty$-category as defined in \cite{Haugseng2017HigherMoritaCategory,Haugseng2023RemarksHigherMorita} and \cite[\S 4.4]{HA}. We believe that \cref{theoremB:Morita} is an interesting result on its own. To the author's knowledge, even in the case of a monoidal $\infty$-category $\cC$, such a universal property was not known. Note that a monoidal $\infty$-category $\cC$ is Moritable precisely if the underlying $\infty$-category of $\cC$ admits geometric realizations and the monoidal product preserves these in both variables separately; these are the usual conditions that ensure the existence and the associativity of the relative tensor product.
	
	\begin{remark}
		An analog of \cref{theoremB:Morita} for ordinary (virtual) double categories was proved by Cruttwell--Shulman \cite[Proposition 5.14]{CruttwellShulman2010UnifiedFrameworkGeneralized}.
	\end{remark}
	
	Unfortunately, the proof of \cref{theoremB:Morita} is quite technical. This paper is therefore split into two parts: In \cref{part1}, we prove \cref{theoremA}, essentially using \cref{theoremB:Morita} as a black box. In \cref{PartII}, we then prove \cref{theoremB:Morita}.
	
	\paragraph{\textbf{Overview of the proof of \cref{theoremA}}}
	
	After treating some background on double $\infty$-categories in \cref{section:double-categories,section:doublespans}, we prove a preliminary straightening equivalence in \cref{section:first-straightening}: we show that for any Segal space $\cC$, there is an equivalence between the $\infty$-category $\Seg(\Spc)_{/\cC}$ of Segal spaces over $\cC$ and the $\infty$-category of lax functors from $\cC$ to a certain double $\infty$-category $\DSpan(\Spc)$ of spans (\cref{theorem:first-straightening}). This result is a straightforward consequence of a universal property of $\DSpan(\Spc)$ established by Haugseng in \cite[Corollary 3.11]{Haugseng2021SegalSpacesSpans}.
	
	In \cref{section:Morita-double-cat}, we show that $\DSpan(\Spc)$ is Moritable and hence that its Morita double $\infty$-category in the sense of \cref{theoremB:Morita} exists. We then prove \cref{theoremA} in \cref{section:second-straightening}: Using the universal property of $\Mor(\DSpan(\Spc))$ and our preliminary straightening result, we obtain a natural equivalence
	\[\Seg(\Spc)_{/\cC} \simeq \UnitLax(\cC, \Mor(\DSpan(\Spc))).\]
	\Cref{theoremA} is then established by restricting to an appropriate full double subcategory $\DCorr$ of $\Mor(\DSpan(\Spc))$.
	
	Various other definitions of double $\infty$-categories of correspondences exist in the literature. We show in \cref{section:comparison} that $\DCorr$ agrees with these.
	
	One may wonder if \cref{theoremA} can be used to deduce straightening equivalences for various types of fibrations between $\infty$-categories. This is indeed the case, and we will work this out for Conduché fibrations and locally cocartesian fibrations in \cref{section:other-types-of-fibrations}; other types of fibrations follow by similar arguments.

	For an overview of the proof of \cref{theoremB:Morita}, the reader is referred to \cref{PartII}.
	
	\paragraph{\textbf{Relation to other work}}
	
	A result very similar to \cref{theoremA} has been obtained by Kositsyn \cite[Corollary 5.21]{Kositsyn2021CompletenessMonadsTheories}. He shows that for a certain flagged $(\infty,2)$-category of correspondences $\twoCorr$, there is an equivalence between $(\Cat_\infty)_{/\cC}$ and a certain $\infty$-category $\mathrm{Hom}^{\mathrm{lax},\mathrm{unit},!}_{\mathrm{Cat}}(\cC,\Corr)$. However, the definition of this $\infty$-category is quite ad hoc and complicated. In particular, it depends on a specific collection of 1-morphisms in $\twoCorr$ which is not intrinsic to this flagged $(\infty,2)$-category and should be treated as extra data. As such, it violates the principle of equivalence. By using double $\infty$-categories, we obtain a much cleaner statement that does not depend on such a choice. Moreover, our proof of \cref{theoremA} is simpler and significantly shorter than Kositsyn's proof.
	
	Shortly after the first version of this paper appeared, an alternative proof of \cref{theoremA} appeared by Heine \cite[Corollary 1.9]{Heine2024LocalglobalPrincipleParametrized}. The idea of his proof is quite different: he shows that for any functor $\cD \to \cC$, one can construct its fiberwise free cocompletion and that this is a locally cartesian fibration. (This is also shown by Lurie in \cite[\href{https://kerodon.net/tag/05K8}{Tag 05K8}]{kerodon}.) Using straightening-unstraightening for locally cartesian fibration, one can then deduce \cref{theoremA}. His construction of the double $\infty$-category of correspondences is quite different, and we compare it to our construction in \cref{section:comparison}.
	
	\paragraph{\textbf{Outlook}} Our main result leads to several natural questions. Let us mention three such questions.
	
	\subparagraph{\textit{An $(\infty,2)$-categorical enhancement}} Observe that the $\infty$-category $(\Cat_{\infty})_{/\cC}$ is naturally an $(\infty,2)$-category: it inherits a $\Cat_\infty$-enrichment from $\Cat_\infty$. The $\infty$-category $\UnitLax(\cC,\DCorr)$ also admits a canonical refinement to an $(\infty,2)$-category. In future work, we will show that the equivalence of \cref{theoremA} refines to a natural equivalence between these $(\infty,2)$-categories.
	
	\subparagraph{\textit{An internal straightening equivalence}} Many of the proofs considered in this paper go through when $\Spc$ is replaced by other $\infty$-categories. It is likely that for a given $\infty$-topos $\cB$, one can obtain a version of \cref{theoremA} for (double) $\infty$-categories \emph{internal} to $\cB$.
	
	\subparagraph{\textit{A version of the main result for $(\infty,n)$-categories}} The proof of \cref{theoremA} uses properties of the category $\Delta$ which are also enjoyed by many other indexing categories: for example, by \cite{Kern2024AllSegalObjects} a version of \cref{theorem:universal-property-span-double-cat} also holds for Joyal's categories $\Theta_n$ which are used to model $(\infty,n)$-categories. If one can prove a version of \cref{theoremB:Morita} where $\Delta$ is replaced by $\Theta_n$ (or any other appropriate category of ``shapes''), then one should be able to deduce a version of \cref{theoremA} for $(\infty,n)$-categories and possibly also $(\infty,\infty)$-categories. Using methods similar to those of \cref{section:other-types-of-fibrations}, one could then deduce straightening-unstraightening equivalences for different kinds of fibrations between $(\infty,n)$-categories or $(\infty,\infty)$-categories. This will be the topic of future joint work with Félix Loubaton and Jaco Ruit.
	
	\paragraph{\textbf{Conventions}} Throughout this paper, we take a model-independent approach to the theory of $\infty$-categories. From now on, we will generally drop ``$\infty$'' from our notation: for example, we will write ``category'' for ``$\infty$-category'' and  $\Cat$ for the ($\infty$-)category of ($\infty$-)categories. Given a category $\cB$ and objects $p \colon \cC \to \cB$ and $q\colon \cD \to \cB$ in $\Cat_{/\cB}$, we will write $\Fun_{/\cB}(\cC,\cD)$ for the category $\Fun(\cC,\cD) \times_{\Fun(\cC,\cB)} \{q\}$.
	
	\paragraph{\textbf{Acknowledgements}} The author wishes to thank Fernando Abellán, Rune Haugseng, Louis Martini and Jaco Ruit for many interesting discussions related to this work. He is also grateful to Jaco Ruit for introducing him to the world of double categories. Furthermore, he wishes to thank Robert Burklund for raising the question of whether every functor can be straightened. Finally, he thanks Nathanael Arkor for providing references to analogs of Theorems \ref{theoremA} and \ref{theoremB:Morita} in ordinary (double) category theory.
	
	The author is grateful to the Copenhagen Centre for Geometry and Topology (DNRF151) and the Max Planck Institute for Mathematics in Bonn for their hospitality during the writing of this paper. During a part of this period he was supported by a grant from the Knut and Alice Wallenberg Foundation.

\subfile{part-I.tex}

\subfile{part-II.tex}

	\printbibliography
	
\end{document}

%% file: preamble.tex
\usepackage[english]{babel}
\usepackage[utf8]{inputenc}
\usepackage[T1]{fontenc}
\usepackage{microtype}
\usepackage[dvipsnames]{xcolor}
\usepackage{subfiles}

\usepackage[textsize = footnotesize, colorinlistoftodos]{todonotes}

\usepackage{eucal}
\usepackage[shortlabels]{enumitem}
\usepackage[super]{nth}

\usepackage[backend=biber,
style=ext-alphabetic,
sorting=nyt,
giveninits=true,
isbn=false,
maxalphanames=4,
date=year,
maxbibnames=99,
backref=false,
articlein=false]{biblatex}
\usepackage{csquotes}

\usepackage{amsmath}
\usepackage{amssymb}

\usepackage{amsthm}
\usepackage{amsfonts}
\usepackage{graphicx}

\usepackage{dsfont} 
\usepackage{mathtools}
\usepackage{xparse}
\usepackage{tikz-cd}
\usepackage{quiver}
\usepackage{nccmath}

\usepackage{tikz}
\usepackage{tabularx}
\newcolumntype{Y}{>{\centering\arraybackslash}X}
\usetikzlibrary{shapes.geometric, arrows}
\newcommand{\nwidth}{4cm}
\tikzstyle{no} = [rectangle, rounded corners, 
minimum width=\nwidth, 
minimum height=1cm,
text width = \nwidth,
align=center,
draw=black]

\tikzstyle{arrowv} = [implies-,double equal sign distance,shorten >=0.15cm,shorten <=0.15cm]
\tikzstyle{arrow} = [implies-,double equal sign distance,shorten >=0.05cm,shorten <=0.05cm]

\usepackage[pdfborder={0 0 0}]{hyperref}
\hypersetup{
	colorlinks=true,
	allcolors=.,
	bookmarksdepth = 3
}
\usepackage[capitalize,nameinlink]{cleveref}
\crefformat{equation}{#2(#1)#3}
\Crefformat{equation}{#2(#1)#3}

\theoremstyle{plain}
\newtheorem{theorem}{Theorem}[section]
\newtheorem{lemma}[theorem]{Lemma}
\newtheorem{proposition}[theorem]{Proposition}
\newtheorem{corollary}[theorem]{Corollary}

\newtheorem*{theorem*}{Theorem}
\newtheorem{theoremA}{Theorem}

\newtheorem*{claim}{Claim}

\theoremstyle{definition}
\newtheorem{definition}[theorem]{Definition}
\newtheorem{remark}[theorem]{Remark}
\newtheorem*{remark*}{Remark}
\newtheorem{example}[theorem]{Example}
\newtheorem{notation}[theorem]{Notation}
\newtheorem*{notation*}{Notation}
\newtheorem{construction}[theorem]{Construction}
\newtheorem{observation}[theorem]{Observation}
\newtheorem{warning}[theorem]{Warning}

\DeclareMathOperator{\id}{id}
\DeclareMathOperator*{\colim}{colim}

\DeclareMathOperator{\Map}{Map}

\DeclareMathOperator{\const}{const}
\DeclareMathOperator{\Twr}{Tw^r}

\DeclareMathOperator{\Ar}{Ar}
\DeclareMathOperator{\Lan}{Lan}
\DeclareMathOperator{\VLan}{VLan}

\DeclareMathOperator{\Sq}{Sq}
\DeclareMathOperator{\Ver}{Vert}

\newcommand{\hide}[1]{}

\newlist{numberenum}{enumerate}{1}
\setlist[numberenum]{\upshape(\arabic*)}

\newcommand{\Fun}{\mathrm{Fun}}
\newcommand{\Lax}{\mathrm{Lax}}

\newcommand{\UnitLax}{\mathrm{Lax}_\mathrm{un}}
\newcommand{\UnitOplax}{\mathrm{Oplax}_\mathrm{un}}

\newcommand{\inert}{\mathrm{inert}}
\newcommand{\idle}{\mathrm{idle}}

\newcommand{\Spc}{\mathcal{S}}

\newcommand{\Cat}{\mathcal{C}\mathrm{at}}

\newcommand{\DblCat}{\mathcal{D}\mathrm{bl}\mathcal{C}\mathrm{at}}
\newcommand{\DblCatlax}{\mathcal{D}\mathrm{bl}\mathcal{C}\mathrm{at}^\mathrm{lax}}
\newcommand{\DblCatulax}{\mathcal{D}\mathrm{bl}\mathcal{C}\mathrm{at}^\mathrm{ulax}}
\newcommand{\coCart}{\mathrm{co}\mathcal{C}\mathrm{art}}
\newcommand{\LoccoCart}{\mathcal{L}\mathrm{oc}\mathrm{co}\mathcal{C}\mathrm{art}}
\newcommand{\Cart}{\mathcal{C}\mathrm{art}}

\DeclareMathOperator{\Env}{Env}
\DeclareMathOperator{\UnitEnv}{Env^\mathrm{un}}

\newcommand{\LaxMor}{\overline{\mathbb{M}\mathrm{or}}}
\newcommand{\Mor}{\mathbb{M}\mathrm{or}}
\newcommand{\op}{\mathrm{op}}

\newcommand{\Seg}{\mathcal{S}\mathrm{eg}}
\newcommand{\CSeg}{\mathcal{CS}\mathrm{eg}}
\newcommand{\un}{\smallint}

\newcommand{\activ}{\mathrm{act}}
\newcommand{\alln}{\Delta^\op_{/[n]}}
\newcommand{\idln}{\Lambda^\op_{/[n]}}
\newcommand{\all}[1]{\Delta^\op_{/[#1]}}
\newcommand{\idl}[1]{\Lambda^\op_{/[#1]}}
\newcommand{\Corr}{\mathcal{C}\mathrm{orr}}
\newcommand{\twoCorr}{\mathrm{Corr}}
\newcommand{\CondFib}{\mathcal{C}\mathrm{ond}}

\newcommand{\ver}{\mathrm{v}}
\newcommand{\hor}{\mathrm{h}}
\newcommand{\hop}{\mathrm{hop}}
\newcommand{\eL}{\mathcal{L}}
\newcommand{\eR}{\mathcal{R}}

\newcommand{\DD}{\mathbb{D}}
\newcommand{\DE}{\mathbb{E}}
\newcommand{\DSpan}{\mathbb{S}\mathrm{pan}}

\newcommand{\DProfSeg}{\mathbb{C}\mathrm{orr}_{\mathrm{Seg}}}
\newcommand{\DProf}{\mathbb{C}\mathrm{orr}}

\newcommand{\DCorr}{\mathbb{C}\mathrm{orr}}
\newcommand{\DSeg}{\mathbb{S}\mathrm{eg}}
\newcommand{\DCatComp}{\mathbb{C}\mathrm{at}}
\newcommand{\DCond}{\mathbb{C}\mathrm{ond}}

\newcommand{\unit}{\mathds{1}}

\newcommand{\cB}{\mathcal{B}}
\newcommand{\cC}{\mathcal{C}}
\newcommand{\cD}{\mathcal{D}}
\newcommand{\cE}{\mathcal{E}}

\newcommand{\cL}{\mathcal{L}}

\newcommand{\cP}{\mathcal{P}}
\newcommand{\cQ}{\mathcal{Q}}
\newcommand{\cR}{\mathcal{R}}

\newcommand{\cX}{\mathcal{X}}

\makeatletter
\newcommand{\oset}[3][0ex]{%
	\mathrel{\mathop{#3}\limits^{
			\vbox to#1{\kern-2\ex@
				\hbox{$\scriptstyle#2$}\vss}}}}
\makeatother
\newcommand{\eqarrow}{\oset[-.16ex]{\sim}{\longrightarrow}}

\tikzset{
	rot90/.style={anchor=south, rotate=90, inner sep=.5mm}
}

\newcommand{\barotimes}{\mathbin{\bar\otimes}}

\makeatletter
\def\slashedarrowfill@#1#2#3#4#5{%
	$\m@th\thickmuskip0mu\medmuskip\thickmuskip\thinmuskip\thickmuskip
	\relax#5#1\mkern-7mu%
	\cleaders\hbox{$#5\mkern-2mu#2\mkern-2mu$}\hfill
	\mathclap{#3}\mathclap{#2}%
	\cleaders\hbox{$#5\mkern-2mu#2\mkern-2mu$}\hfill
	\mkern-7mu#4$%
}
\def\rightslashedarrowfill@{%
	\slashedarrowfill@\relbar\relbar\mapstochar\rightarrow}
\newcommand\xprofun[2][]{%
	\ext@arrow 0055{\rightslashedarrowfill@}{#1}{#2}}
\makeatother

\NewDocumentCommand\derprojlim{e{_}}{\mathchoice
	{\varprojlim  \IfValueT{#1}{_{\mathclap{#1}}}{}^{\!1\!}\mathop{}}
	{\varprojlim^1  \IfValueT{#1}{_{#1}}}
	{\varprojlim^1  \IfValueT{#1}{_{#1}}}
	{\varprojlim^1  \IfValueT{#1}{_{#1}}}}

\newcommand{\lax}{\mathrm{lax}}
\newcommand{\oplax}{\mathrm{oplax}}

\newcommand{\thmqed}[1]{\pushQED{\qed} #1 \qedhere\popQED}

\renewcommand{\subset}{\subseteq}

%% file: part-I.tex
\part{Straightening every functor}\label{part1}

\section{Background on double categories}\label{section:double-categories}

In this section, we give a brief introduction to the theory of double categories. For more details, the reader is referred to the excellent introduction \cite[\S 2]{Ruit2023FormalCategoryTheory}.

Recall that $\Delta$ is the category of finite non-empty linearly ordered sets. We will write $[n]$ for the linearly ordered set $\{0 < 1 < \cdots < n\}$ with $n+1$ elements. A map in $\Delta$ will be called \emph{idle} if it is a surjection onto a convex subset and \emph{inert} if it is furthermore injective. Observe that up to (unique) isomorphism, any inert map is a subinterval inclusion of the form $[i,j] \coloneqq \{i,i+1,\cdots,j\} \hookrightarrow [n]$; we will denote this inclusion by $\rho_{i,j}$. Finally, a map will be called \emph{active} if it preserves the minimal and maximal elements.

\begin{definition}[Segal objects]
	Let $\cC$ be a category. A functor $X_\bullet \colon \Delta^\op \to \cC$ is called a \emph{Segal object in $\cC$} if for every $n \geq 1$, the map
	\[X_n \to X_1 \times_{X_0} \cdots \times_{X_0} X_1\]
	induced by the maps $\rho_{i,i+1}$ and $\rho_i$ is an equivalence. It is called complete if moreover the map
	\[X_0 \to X_3 \times_{X_1 \times X_1} (X_0 \times X_0)\]
	is an equivalence. The full subcategory of $\Fun(\Delta^\op,\cC)$ spanned by the Segal objects is denoted $\Seg(\cC)$.
\end{definition}

\begin{example}[Complete Segal spaces]
	A (complete) Segal space is a (complete) Segal object in the category $\Spc$ of spaces. By \cite{JoyalTierney2007QuasicategoriesVsSegal}, the category $\CSeg(\Spc)$ of complete Segal spaces is equivalent to $\Cat$.
\end{example}
		
\begin{definition}[Double categories]
	A \emph{double category} is a Segal object $\DD_\bullet$ in $\Cat$. A morphism in $\Fun(\Delta^\op,\Cat)$ between double categories will be called a \emph{double functor}, or \emph{functor} for short, and the category of double categories will be denoted $\DblCat$.
\end{definition}

\begin{remark}
	Some authors also require that a double category be complete. We will not use this convention, since our main example, the double category $\DProf$ of correspondences defined in \cref{definition:double-cat-of-profunctors}, is \emph{not} complete.
\end{remark}

\begin{example}[Monoidal categories]\label{example:Moniodal-as-double}
	A \emph{monoidal category} is a double category $\DD$ such that $\DD_0$ is contractible (cf.\ \cite[Definition 4.1.3.6]{HA}).
\end{example}

\begin{example}[2-categories]\label{example:2-categories}
	A \emph{2-category} is a double category $\DD$ that is complete and such that $\DD_0$ is a space. (In the literature, these are often called \emph{complete 2-fold Segal spaces}.)
\end{example}

\begin{example}[Horizontal embedding]\label{example:Horizontal-embedding}
	Any Segal space $\cC \colon \Delta^\op \to \Spc$ can be viewed as a double category by postcomposing with the inclusion $\Spc \hookrightarrow \Cat$. We will write $\cC^\hor$ for the resulting double category and call it the \emph{horizontal embedding} of $\cC$.
\end{example}

\begin{example}[Vertical embedding]\label{example:Vertical-embedding}
Given a category $\cC$, one can consider the constant simplicial object $\Delta^\op \to \Cat$ with value $\cC$. This object will be denoted $\cC^\ver$ and called the \emph{vertical embedding} of $\cC$. Observe that any constant simplicial object satisfies the Segal condition, hence this defines a fully faithful functor $(-)^\ver \colon \Cat \to \DblCat$.
\end{example}

\begin{definition}[Mapping categories]
	Let $\DD$ be a double category and let objects $x,y \in \DD_0$ be given. Then the \emph{mapping category} $\DD(x,y)$ is defined by the pullback
	\[\begin{tikzcd}
		\DD(x,y) \ar[r," "] \ar[dr,"\lrcorner",very near start, phantom] \ar[d," "] & \DD_1 \ar[d,"{(d_1,d_0)}"]\\
		* \ar[r,"{(x,y)}"] & \DD^{\times 2}_0.
	\end{tikzcd}\]
	Objects in $\DD(x,y)$ will be called \emph{horizontal morphisms} or \emph{horizontal arrows}.
\end{definition}

\begin{example}[Horizontal opposite]\label{example:Horizontal-opposite}
	Observe that the functor $(-)^\op \colon \Cat \to \Cat$ restricts to a functor $(-)^\op \colon \Delta \to \Delta$. Given a double category $\DD$, its \emph{horizontal opposite} $\DD^\hop$ is defined as the composition of the functors $(-)^\op \colon \Delta^\op \to \Delta^\op$ and $\DD \colon \Delta^\op \to \Cat$. Observe that $\DD^\hop$ is again a double category and that its mapping categories are given by $\DD^\hop(x,y) \simeq \DD(y,x)$.
\end{example}

\begin{definition}[Horizontal composition]
	Let $\DD$ be a double category and let objects $x,y,z \in \DD_0$ be given. By pulling back the map
	\[\DD_1 \times_{\DD_0} \DD_1 \simeq \DD_2 \xrightarrow{d_1} \DD_1\]
	along $(x,y,z) \colon * \to \DD_0^{\times 3}$, we obtain the \emph{horizontal composition functor}
	\[\DD(x,y) \times \DD(y,z) \to \DD(x,z).\]
	The degeneracy map $\DD_0 \to \DD_1$ yields, for every $x$ in $\DD_0$, a \emph{horizontal identity morphism} $\unit_x \in \DD(x,x)$.
\end{definition}

Given a double category $\DD$, one can define subobjects of $\DD$ by appropriately choosing full subcategories of the categories $\DD_n$.

\begin{construction}[Locally full double subcategories]\label{definition:locally-full-subcategory}
	Let $\DD$ be a double category and let $\DE_0 \subset \DD_0$ and $\DE_1 \subset \DD_1$ be full subcategories. Suppose that
	\begin{enumerate}[(a)]
		\item the source-target projection $\DE_1 \to \DD_0^{\times 2}$ lands in $\DE_0^{
		\times 2}$,
		\item the degeneracy map $s \colon \DE_0 \to \DD_1$ lands in $\DE_1$, and
		\item the horizontal composition $\DD_1 \times_{\DD_0} \DD_1 \simeq \DD_2 \to \DD_1$ takes $\DE_1 \times_{\DE_0} \DE_1$ to $\DE_1$.
	\end{enumerate}
	Define $\DE_n$ to be the full subcategory of $\DD_n$ corresponding to $\DE_1 \times_{\DE_0} \cdots \times_{\DE_0} \DE_1$ under the equivalence $\DD_n \simeq \DD_1 \times_{\DD_0} \cdots \times_{\DD_0} \DD_1$. Then for any $\sigma \colon [n] \to [m]$ in $\Delta$, the map $\sigma^* \colon \DD_m \to \DD_n$ restricts to a map $\DE_m \to \DE_n$, hence $\DE_\bullet$ forms a subobject of $\DD$. A double category $\DE$ constructed in this way will be called a \emph{locally full (double) subcategory}. If $\DE_1$ consists of \emph{all} objects of $\DD_1$ whose image under the source-target projection lands in $\DE_0^{\times 2}$, then we call $\DE$ a \emph{full (double) subcategory} of $\DD$.
\end{construction}

Since the mapping categories $\DD(x,y)$ of a double category $\DD$ are \emph{categories} and not merely \emph{spaces}, one can loosen the definition of a double functor by requiring that horizontal composition is only preserved up to a \emph{non-invertible} morphism in $\DD(x,y)$. A convenient way to make this precise is as follows.

\begin{definition}[Lax functors]\label{definition:lax-functors}
	Let $\DD, \DE \colon \Delta^\op \to \Cat$ be double categories and let $\un \DD, \un \DE \to \Delta^\op$ denote their cocartesian unstraightenings. A \emph{lax functor} $F$ from $\DD$ to $\DE$ is a commutative triangle
	\[\begin{tikzcd}[column sep=tiny]
		\int \DD \ar[rr,"F'"] \ar[dr] && \int \DE \ar[dl] \\
		& \Delta^\op &
	\end{tikzcd}\]
	such that $F'$ preserves cocartesian lifts of inerts. If $F'$ moreover preserves cocartesian lifts of idle maps, then $F$ is called a \emph{unital lax functor}.
	
	We write $\Lax(\DD,\DE)$ for the full subcategory of $\Fun_{/\Delta^\op}(\un \DD,\un \DE)$ spanned by the lax functors, $\UnitLax(\DD,\DE)$ for the full subcategory of unital lax functors and $\Fun(\DD,\DE)$ for the full subcategory spanned by those functors that preserve all cocartesian lifts. The categories of double categories and (unital) lax functors between them will be denoted $\DblCatlax$ and $\DblCatulax$.
\end{definition}

\begin{remark}
	In \cref{PartII}, we will also use the notation $\Lax(\DD,\DE)$, $\UnitLax(\DD,\DE)$ and $\Fun(\DD,\DE)$ when $\DD$ and/or $\DE$ are not double categories---that is, when $\DD$ and $\DE$ are functors $\Delta^\op \to \Cat$ that do not necessarily satisfy the Segal condition.
\end{remark}

\begin{remark}
	In \cref{lemma:Unital-lax-envelope}, we relate unital lax functors in the sense of \cref{definition:lax-functors} to the orientals defined in \cite{Street1987AlgebraOrientedSimplexes}.
	This can be seen as a justification for the definition of unital lax functors given here.
\end{remark}

\begin{lemma}\label{lem:double-functors-equivalence}
	Let $F,G \colon \DD \to \DE$ be functors between double categories and let $\nu \colon  F \Rightarrow G$ be a morphism in $\Fun(\DD,\DE)$.
	Restricting to the fiber over $[1]$ in $\Delta^\op$, we obtain a natural transformation $\nu_1 \colon F_1 \Rightarrow G_1$ in $\Fun(\DD_1, \DE_1)$.
	Then $\nu \colon F \Rightarrow G$ is an equivalence if and only if $\nu_1 \colon F_1 \Rightarrow G_1$ is an equivalence.
\end{lemma}

\begin{proof}
	We need to show that $\nu_n \colon F_n \Rightarrow G_n$ is an equivalence for every $n \geq 0$.
	When $n=0$, this follows since $[0]$ is a retract of $[1]$ in $\Delta$.
	For $n > 1$, this follows from the Segal condition.
\end{proof}

Throughout this paper, we will also have use for a slightly more general notion of Segal objects. Suppose $\cD \to \Delta^\op$ is a functor that has cocartesian lifts of all inert morphisms. Given an object $d$ of $\cD$ lying above $[n]$, we can then pick cocartesian lifts $d \to d_{i,j}$ of the morphisms $\rho_{i,j} \colon [i,j] \hookrightarrow [n]$. This allows us to express a Segal condition for functors out of $\cD$.

\begin{definition}[Generalized Segal objects]\label{definition:generalized-Segal-objects}
	Let $\cD \to \Delta^\op$ be a functor that has cocartesian lifts of inert morphisms and let $\cC$ be a category that admits pullbacks. For any functor $F \colon \cD \to \cC$ and any object $d \in \cD$ lying over $[n]$, we obtain the \emph{Segal map (of $F$ at $d$)}
	\[F(d) \to F(d_{0,1}) \times_{F(d_{1,1})} \cdots \times_{F(d_{n-1,n-1})} F(d_{n-1,n})\]
	by taking cocartesian lifts of the maps $\rho_{i,i+1}$ and $\rho_{i,i}$. We say that a functor $F \colon \cD \to \cC$ \emph{satisfies the Segal condition} if for every $d$ in $\cD$, the Segal map of $F$ at $d$ is an equivalence. The category of such functors $\cD \to \cC$ will be denoted $\Seg_{\cD}(\cC)$. If $\cD$ is of the form $\un \DD$ for some double category $\DD$, then we will simply write $\Seg_{\DD}(\cC)$ for $\Seg_{\un \DD}(\cC)$.
\end{definition}

Observe that the construction of the category $\Seg_{\cD}(\cC)$ is functorial in both $\cD$ and $\cC$ if we restrict to functors that preserve cocartesian lifts of inerts and pullbacks, respectively.

\section{The double categories of spans}\label{section:doublespans}

We will now consider an important example of a double category, namely the span double category.

\begin{construction}[The span double category]
	Let $\Sigma^n$ denote the poset of convex subsets of $[n]$, ordered by reverse inclusion. An order-preserving map $\alpha \colon [n] \to [m]$ defines a map $\alpha_! \colon \Sigma^n \to \Sigma^m$ by sending $[i,j]$ to $[\alpha(i),\alpha(j)]$, making $\Sigma^\bullet$ into a functor $\Delta \to \Cat$. Given a category $\cC$, let $\DSpan(\cC)_n$ be the full subcategory of $\Fun(\Sigma^n,\cC)$ spanned by those functors $F \colon \Sigma^n \to \cC$ for which all squares of the form
	\[\begin{tikzcd}
		F[i,j] \ar[r," "] \ar[d," "] & F[i,j-1] \ar[d," "]\\
		F[i+1,j] \ar[r," "] & F[i+1,j-1]
	\end{tikzcd}\]
	are cartesian. It follows from the pasting lemma for pullback squares that restriction along any $\alpha_! \colon \Sigma^n \to \Sigma^m$ preserves functors with this property. In particular, the categories $\DSpan(\cC)_n$ assemble into a simplicial object $\DSpan(\cC)$ in $\Cat$, which we call the \emph{span double category of $\cC$}.
\end{construction}

It is proved in \cite[Proposition 5.14]{Haugseng2018IteratedSpansClassical} that $\DSpan(\cC)$ is indeed a double category.

\begin{example}[Mapping categories in $\DSpan(\cC)$]\label{example:Mapping-categories-DSpan}
	Observe that $\Sigma^1$ is isomorphic to the left cone $\{0,1\}^\triangleleft$. Moreover, the source-target projection $\DSpan(\cC) \to \cC^{\times 2}$ is obtained by restriction along the inclusion $\{0,1\} \hookrightarrow \{0,1\}^\triangleleft$. In particular, for any pair of objects $c,d \in \cC$, the mapping category $\DSpan(\cC)(c,d)$ is given by the overcategory $\cC_{/c,d}$. Observe that this agrees with $\cC_{/c \times d}$ if $\cC$ admits binary products. Moreover, it is not hard to verify that the horizontal composition functor $\DSpan(\cC)(c,d) \times \DSpan(\cC)(d,e) \to \DSpan(\cC)(c,e)$ is given by
	\[\left(
		\begin{tikzcd}[column sep={2em,between origins}, row sep={2em,between origins}]
			& x \ar[dr] \ar[dl] &\\
			c && d
		\end{tikzcd},
		\begin{tikzcd}[column sep={2em,between origins}, row sep={2em,between origins}]
			& y \ar[dr] \ar[dl] &\\
			d && e
		\end{tikzcd}
	\right)
	\begin{tikzcd}
		 \;\ar[r, maps to] & \;
	\end{tikzcd}
	\begin{tikzcd}[column sep={2em,between origins}, row sep={2em,between origins}]
		& x \times_d y \ar[dr] \ar[dl] &\\
		c && e
	\end{tikzcd}\]
\end{example}

Given functors $\cD,\cE \to \Delta^\op$ that admit cocartesian lifts of inerts, we will write $\Fun^\inert_{/\Delta^\op}(\cD,\cE)$ for the full subcategory of $\Fun_{/\Delta^\op}(\cD,\cE)$ of functors that preserve cocartesian lifts of inerts. By \cite[Corollary 3.11]{Haugseng2021SegalSpacesSpans}, the span double category has the following universal property for lax functors into it:

\begin{theorem}[The universal property of $\DSpan(\cC)$]\label{theorem:universal-property-span-double-cat}
	Let $\cC$ be a category that admits pullbacks and $\cD \to \Delta^\op$ a functor that admits cocartesian lifts of inerts. Then there is an equivalence of categories
	\[\Fun^\inert_{/\Delta^\op}(\cD,\DSpan(\cC)) \eqarrow \Seg_{\cD}(\cC)\]
	that is natural in both $\cC$ and $\cD$. In particular, if $\DD$ is a double category, then there is a natural equivalence
	\[\thmqed{\Lax(\DD,\DSpan(\cC)) \simeq \Seg_{\DD}(\cC).}\]
\end{theorem}

\begin{remark}
	In \cref{theorem:universal-property-span-double-cat}, we only assume that $\cD \to \Delta^\op$ has cocartesian lifts of inerts. This is a weaker condition than the one in Corollary 3.11 of \cite{Haugseng2021SegalSpacesSpans}, where it is assumed that $\cD$ is a generalized planar $\infty$-operad. However, the proof of loc.\ cit.\ shows that this weaker condition is sufficient.
\end{remark}

\section{A first straightening result}\label{section:first-straightening}

We will now prove our first straightening result: every functor $\cD \to \cC$ can be straightened to a lax functor $\cC^\hor \to \DSpan(\Spc)$. It turns out that this result holds already for Segal spaces that are not necessarily complete.

\begin{theorem}\label{theorem:first-straightening}
	Let $\cC$ be a Segal space. Then there is a natural equivalence of categories
	\[\Seg(\Spc)_{/\cC} \simeq \Lax(\cC^\hor,\DSpan(\Spc)).\]
\end{theorem}

\begin{proof}[Proof of \cref{theorem:first-straightening}]
	Let $\cC \colon \Delta^\op \to \Spc$ be a Segal space. By cocartesian straightening, we obtain the left fibration $\int \cC = \Delta_{/\cC}^\op \to \Delta^\op$. By \cite[Prop.\ 7.8]{Haugseng2015RectificationEnrichedCategories}, it follows that left Kan extension along $\Delta_{/\cC}^\op \to \Delta^\op$ induces an equivalence
	\[\Fun(\Delta_{/\cC}^\op,\Spc) \eqarrow \Fun(\Delta^\op, \Spc)_{/\cC}.\]
	Naturality in $\cC$ follows as in the proof of \cite[Theorem 7.5]{Haugseng2015RectificationEnrichedCategories}.
	
	By an argument similar to \cite[Prop.\ 7.9]{Haugseng2015RectificationEnrichedCategories}, it follows that this map restricts to an equivalence
	\[\Seg_{\cC^\hor}(\Spc) \simeq \Seg(\Spc)_{/\cC}.\]
	Combined with \cref{theorem:universal-property-span-double-cat}, the desired equivalence follows.
\end{proof}
	
\section{The Morita double category}\label{section:Morita-double-cat}

If $\cC$ is a monoidal category, then under mild assumptions on $\cC$ one can construct its \emph{Morita category}. The objects of this Morita category are algebras in $\cC$, while its morphisms are bimodules and composition is given by the relative tensor product. Under suitable assumptions, this construction can be generalized to double categories: if $\DD$ is a \emph{Moritable} double category (see \cref{def:Moritable}), then one can construct its \emph{Morita double category} $\Mor(\DD)$. This double category moreover satisfies the following universal property.

\begin{theorem}\label{theorem:Morita-category-existence-and-universal-property}
	Let $\DD$ be a Moritable double category in the sense of \cref{def:Moritable}. Then there exists a double category $\Mor(\DD)$, its \emph{Morita double category}, satisfying the universal property that for any double category $\DE$, there is a natural equivalence
	\[\Lax(\DE,\DD) \simeq \UnitLax(\DE,\Mor(\DD)).\]
\end{theorem}

Since the proof of this theorem is quite long and technical, we give it in \cref{PartII} below and use it as a black box for now.

The proof of \cref{theoremA} follows by applying this result to the span double category $\DSpan(\Spc)$. For this, we need to know that $\DSpan(\Spc)$ is Moritable.

\begin{proposition}\label{proposition:DSpan-is-Moritable}
	Let $\cC$ be a category that admits pullbacks and geometric realizations, and suppose that for any map $f \colon c \to d$ in $\cC$, the pullback functor $f^* \colon \cC_{/d} \to \cC_{/c}$ preserves geometric realizations. Then $\DSpan(\cC)$ is Moritable in the sense of \cref{def:Moritable}.
\end{proposition}

\begin{proof}
	By \cref{example:Mapping-categories-DSpan}, the mapping category $\DSpan(\cC)(c,d)$ is of the form $\cC_{/c,d}$. It follows from \cite[Proposition 1.2.13.8]{HTT} that these categories admit geometric realizations and that these can be computed in $\cC$. Moreover, since the composition $\cC_{/c,d} \times \cC_{/d,e} \to \cC_{/c,e}$ is given by taking pullbacks, our assumptions ensure that the composition maps of $\DSpan(\cC)$ preserve geometric realizations in both variables separately. Finally, the source-target projection $\DSpan(\cC)_1 \to \cC \times \cC$ is a bicartesian fibration, where the cocartesian lifts are arrows in $\DSpan(\cC)_1$ of the form
	\[\begin{tikzcd}
		c & a & d \\
		c' & a' & d'.
		\arrow[from=1-2, to=1-1]
		\arrow["\alpha", "\sim" rot90, from=1-2, to=2-2]
		\arrow[from=1-2, to=1-3]
		\arrow[from=1-1, to=2-1]
		\arrow[from=2-2, to=2-1]
		\arrow[from=2-2, to=2-3]
		\arrow[from=1-3, to=2-3]
	\end{tikzcd}\]
	with $\alpha$ invertible, and the cartesian lifts are the arrows of the form
	\[\begin{tikzcd}
		c & {c \times_{c'} a' \times_{d'} d} & d \\
		c' & a' & d'.
		\arrow[from=1-2, to=1-1]
		\arrow[from=1-2, to=2-2]
		\arrow[from=1-2, to=1-3]
		\arrow[from=1-1, to=2-1]
		\arrow[from=2-2, to=2-1]
		\arrow[from=2-2, to=2-3]
		\arrow[from=1-3, to=2-3]
	\end{tikzcd}\]
\end{proof}

\begin{definition}[The double category of correspondences]\label{definition:double-cat-of-profunctors}
	Let $\cC$ be a category satisfying the assumptions of \cref{proposition:DSpan-is-Moritable}. Then the double category $\Mor(\DSpan(\cC))$ will be denoted $\DProfSeg(\cC)$ and called the \emph{double category of Segal objects and correspondences in $\cC$}. Observe that
	\[\DProfSeg(\cC)_0 \simeq \Lax([0]^\hor,\DSpan(\cC)) \simeq \Seg(\cC)\]
	by \cref{theorem:Morita-category-existence-and-universal-property,theorem:universal-property-span-double-cat}. In particular, we can define the full double subcategory $\DProf(\cC)$ of $\DProfSeg(\cC)$ spanned by the complete Segal objects, which will be called the \emph{double category of internal categories and correspondences in $\cC$}.
\end{definition}

Note that there are other possible definitions of the double category of correspondences, see for example \cite{AyalaFrancis2020FibrationsInftyCategories}. In \cref{section:comparison} below, we will compare our definition with the more classical ones.

\section{The second straightening result}\label{section:second-straightening}

\Cref{theoremA} now readily follows by combining the previous results. Let us write $\DProfSeg$ and $\DProf$ for the double categories $\DProfSeg(\Spc)$ and $\DProf(\Spc)$, respectively.

\begin{theorem}\label{theorem:2nd-straightening-result-noncomplete}
	Let $\cC$ be a Segal space. Then there is an equivalence of categories
	\[\Seg(\Spc)_{/\cC} \simeq \UnitLax(\cC^\hor,\DProfSeg)\]
	that is natural in $\cC$.
\end{theorem}

\begin{proof}
	Combine \cref{theorem:first-straightening,theorem:Morita-category-existence-and-universal-property}.
\end{proof}

Restricting to complete Segal spaces, we obtain our main result.

\begin{corollary}[{\cref{theoremA}}]\label{corollary:2nd-straightening-result-complete}
	Let $\cC$ be a category. Then there is a natural equivalence
	\[\Cat_{/\cC} \simeq \UnitLax(\cC^\hor,\DProf).\]
\end{corollary}

\begin{proof}
	It suffices to show that if $\cC$ is a complete Segal space and $p \colon \cD \to \cC$ a map of Segal spaces, then $\cD$ is complete if and only if for any $c \in \cC_0$, the fiber $p^{-1}(c)$ is a complete Segal space. This is an exercise for the reader.
\end{proof}

This result also implies the following general method for writing down unital lax functors into $\DProf$, which we state because it might be of independent interest. Let us write $\Cat_{/[\bullet]}$ for the functor $[n] \mapsto \Cat_{/[n]}$, where the functoriality is given by pullback.

\begin{proposition}\label{proposition:Stronger-universal-property-DCorr}
	Let $\DD$ be a double category. Then there is a natural equivalence
	\[\UnitLax(\DD,\DProf) \simeq \Fun(\DD,\Cat_{/[\bullet]}).\]
\end{proposition}

\begin{remark}
	Here $\Fun(\DD,\Cat_{/[\bullet]})$ denotes the mapping category in $\Fun(\Delta^\op, \Cat)$. This may be defined as the full subcategory of $\Fun_{/\Delta^\op}(\un \DD, \un \Cat_{/[\bullet]})$ spanned by those functors that preserve all cocartesian lifts.
\end{remark}

\begin{proof}[Proof of \cref{proposition:Stronger-universal-property-DCorr}]
	By \cref{corollary:2nd-straightening-result-complete}, there are natural equivalences
	\[\Cat_{/[\bullet]} \simeq \UnitLax([\bullet]^\hor, \DCorr) \simeq \Fun^{\idle}_{/\Delta^\op}(\Delta_{/[\bullet]}^\op, \un \DCorr).\]
	By \cref{proposition:Partially-cofree-fibrations}, this simplicial object is equivalent to $R_\mathrm{idle}(\un \DCorr)$, for which there is a natural equivalence
	\[\Fun(\DD,R_\mathrm{idle}(\un \DCorr)) \simeq \UnitLax(\DD,\DCorr). \qedhere\]
\end{proof}

\begin{remark}
	In fact, it follows from \cref{proposition:Partially-cofree-fibrations} that we don't need to assume that $\DD$ is a double category in \cref{proposition:Stronger-universal-property-DCorr}: the result holds for any object $\DD$ in $\Fun(\Delta^\op,\Cat)$.
\end{remark}

\section{Comparison to other double categories of correspondences}\label{section:comparison}

Throughout the literature, alternative definitions of the double category $\DProf$ have been proposed. In this section, we will show that our definition is equivalent to those of Ayala--Francis \cite{AyalaFrancis2020FibrationsInftyCategories}, Ruit \cite{Ruit2023FormalCategoryTheory} and Heine \cite{Heine2024LocalglobalPrincipleParametrized}.

We start by recalling Ruit's construction. For any category $\cE$ satisfying the assumptions of \cref{proposition:DSpan-is-Moritable}, Ruit \cite[\S 4]{Ruit2023FormalCategoryTheory} defines a double category $\DSeg(\cE)$ of \emph{categorical objects in $\cE$}. The category $\DSeg(\cE)_n$ is defined as the full subcategory of $\Fun(\alln,\cE)$ spanned by those objects that satisfy the Segal condition (in the sense of \cref{definition:generalized-Segal-objects}) and a condition similar to \cref{definition:composite-objects}. Since $\DSeg(\cE)_0 = \Seg(\cE)$, one can restrict to the full double subcategory spanned by the complete Segal objects; this double category is denoted $\DCatComp(\cE)$ in \cite[\S 4]{Ruit2023FormalCategoryTheory}. 
In his construction, Ruit considers profunctors instead of correspondences as horizontal morphisms; by definition, a \emph{profunctor} from $\cC$ to $\cD$ is a correspondence from $\cD$ to $\cC$. In particular, to compare his construction to ours, we will need to consider the \emph{horizontal opposite} $\DSeg(\cE)^\hop$ in the sense of \cref{example:Horizontal-opposite}.

\begin{proposition}\label{proposition:Equivalence-DCorr-Ruit}
	Let $\cC$ be a category satisfying the assumptions of \cref{proposition:DSpan-is-Moritable}. Then there is a natural equivalence between the double categories $\DProfSeg(\cC)$ of \cref{definition:double-cat-of-profunctors} and $\DSeg(\cC)^\hop$ of \cite[Construction 4.12]{Ruit2023FormalCategoryTheory}. This equivalence restricts to an equivalence between $\DProf(\cC)$ and $\DCatComp(\cC)^\hop$.
\end{proposition}

\begin{proof}
	Let $i$ denote the inclusion $\idln \hookrightarrow \alln$. By \cref{theorem:universal-property-span-double-cat}, we have equivalences
	\[\begin{tikzcd}
		\Fun^\inert_{/\Delta^\op}(\alln, \DSpan(\cC)) \ar[r,"\sim"] \ar[d,"i^*"] & \Seg_{\alln}(\cC) \ar[d,"i^*"]\\
		\Fun^\inert_{/\Delta^\op}(\idln, \DSpan(\cC)) \ar[r,"\sim"] & \Seg_{\idln}(\cC)
	\end{tikzcd}\]
	where the top equivalence is natural in $[n]$. By \cref{lem:composites-idle-char} and \cite[Proposition 4.4 \& Lemma 4.8]{Ruit2023FormalCategoryTheory}, the vertical functors $i^*$ have fully faithful left adjoints $i_!$, and the categories $\DProfSeg(\cC)_n = \Mor(\DSpan(\cC))_n$ and $\DSeg(\cC)_n^\hop$ are defined as the essential images of $i_!$. In particular, we obtain an equivalence $\DProfSeg(\cC)_n \simeq \DSeg(\cC)_n^\hop$ that is natural in $n$, which is exactly what we needed to show. The equivalence $\DProf(\cC) \simeq \DCatComp(\cC)^\hop$ follows by restricting to complete Segal objects.
\end{proof}

We now describe the construction of Ayala--Francis \cite{AyalaFrancis2020FibrationsInftyCategories}. Recall that classically, an ($\infty$-categorical) correspondence from $\cC$ to $\cD$ is defined as a functor $\cC^\op \times \cD \to \Spc$. The composite of two correspondences $F$ from $\cC$ to $\cD$ and $G$ from $\cD$ to $\cE$ is then defined as the coend
\begin{equation}\label{equation:composition-profunctors-coend}
	\medint\int^{d \in \cD} F(-,d) \times G(d,-) = \colim_{(d' \to d) \in \Twr(\cD)} F(-,d') \times G(d,-).
\end{equation}
A correspondence from $\cC$ to $\cD$ can equivalently be described as a functor $p \colon \cE \to [1]$ equipped with equivalences $\cC \simeq p^{-1}(0)$ and $\cD \simeq p^{-1}(1)$---given $p \colon \cE \to [1]$, we recover a correspondence $p^{-1}(0)^\op \times p^{-1}(1) \to \Spc$ via the formula $(c,d) \mapsto \Map_{\cE}(c,d)$. This viewpoint is used by Ayala--Francis \cite{AyalaFrancis2020FibrationsInftyCategories}: they define a Segal space $\Corr_\bullet$ by defining $\Corr_n$ to be the space $\CondFib_{[n]}^\simeq$ of Conduché fibrations over $[n]$. Since any map to $[1]$ is a Conduché fibration, we see that the (horizontal) morphisms of $\Corr_\bullet$ can indeed be viewed as correspondences. Moreover, it follows from item 4 of Lemma 2.2.8 of \cite{AyalaFrancis2020FibrationsInftyCategories} that composition in this Segal space is indeed given by the coend formula of \cref{equation:composition-profunctors-coend}. If one replaces the \emph{space} $\CondFib_{[n]}^\simeq$ by the \emph{category} $\CondFib_{[n]}$ in this construction, one obtains a double category that we will denote by $\DCond$. This gives an alternative construction of a double category of correspondences.

Ruit proves in \cite[Proposition 4.21]{Ruit2023FormalCategoryTheory} that $\DCond$ is equivalent to $\DCatComp(\Spc)^\hop$, so it follows from \cref{proposition:Equivalence-DCorr-Ruit} that $\DCond$ is equivalent to $\DCorr$. We give an alternative proof of the equivalence $\DCond \simeq \DCorr$ here since it is relatively short and conceptually enlightening.

\begin{proposition}\label{proposition:Equivalence-DCorr-AF}
	There is an equivalence $\DCond \simeq \DCorr$ of double categories.
\end{proposition}

\begin{proof}
	By \cref{definition:double-cat-of-profunctors} and \cref{remark:Composites-left-adjoint-to-lax-Segal-map}, $\DProfSeg$ is defined in degree $n$ as the full subcategory of $\LaxMor(\DSpan(\Spc))_n$ spanned by those objects that are in the essential image of the fully faithful left adjoint to
	\[\LaxMor(\DSpan(\Spc))_n \to \LaxMor(\DSpan(\Spc))_1 \times_{\LaxMor(\DSpan(\Spc))_0} \cdots \times_{\LaxMor(\DSpan(\Spc))_0} \LaxMor(\DSpan(\Spc))_1.\]
	Restricting to \emph{complete} Segal spaces and using the equivalence from \cref{proposition:Stronger-universal-property-DCorr}, we see that $\DCorr_\bullet$ is equivalent to the subobject of $\Cat_{/[\bullet]}$ which in degree $n$ is given by the essential image of the left adjoint to
	\[\Cat_{/[n]} \to \Cat_{/[1]} \times_{\Cat} \cdots \times_{\Cat} \Cat_{/[1]}.\]
	Since this left adjoint is given by
	\[\left(\begin{tikzcd}[column sep=tiny, row sep=scriptsize]
		{\cC_{0,1}} & {\cC_1} & \cdots & {\cC_{n-1}} & {\cC_{n-1,n}} \\
		{[0,1]} & {\{1\}} & \cdots & {\{n-1\}} & {[n-1,n]}
		\arrow[from=1-1, to=2-1]
		\arrow[from=1-2, to=1-1]
		\arrow[from=1-2, to=1-3]
		\arrow[from=1-2, to=2-2]
		\arrow[from=1-4, to=1-3]
		\arrow[from=1-4, to=1-5]
		\arrow[from=1-4, to=2-4]
		\arrow[from=1-5, to=2-5]
		\arrow[from=2-2, to=2-1]
		\arrow[from=2-2, to=2-3]
		\arrow[from=2-4, to=2-3]
		\arrow[from=2-4, to=2-5]
	\end{tikzcd}\right) \longmapsto \begin{tikzcd}[column sep=tiny, row sep=scriptsize]
	{\cC_{0,1} \sqcup_{\cC_{1}} \cdots \sqcup_{\cC_{n-1}} \cC_{n-1,n}} \\
	{[n]}
	\arrow[from=1-1, to=2-1]
	\end{tikzcd},\]
	it follows from \cite[Corollary 2.2.13]{AyalaFrancis2020FibrationsInftyCategories} that its essential image is given by $\CondFib_{[n]}$. But this subobject of $\Cat_{/[\bullet]}$ is precisely the definition of $\DCond$.
\end{proof}

\begin{remark}\label{remark:Nonlax-universal-property-DCorr}
	The proof above shows that the equivalence $\UnitLax(\DD,\DProf) \simeq \Fun(\DD,\Cat_{/[\bullet]})$ from \cref{proposition:Stronger-universal-property-DCorr} restricts to an equivalence
	\[\Fun(\DD,\DProf) \simeq \Fun(\DD,\CondFib_{[\bullet]}).\]
\end{remark}

We also immediately obtain a comparison with the double category of correspondences defined by Heine \cite{Heine2024LocalglobalPrincipleParametrized}.

\begin{corollary}
	Let $\mathrm{CORR}$ denote the double category of correspondences defined by Heine \cite[Notation 7.9.(3)]{Heine2024LocalglobalPrincipleParametrized}. Then there is an equivalence $\DCorr \simeq \mathrm{CORR}$ of double categories.
\end{corollary}

\begin{proof}
	Heine proves a natural equivalence
	\[\Fun(\cC^\hor,\mathrm{CORR}) \simeq \CondFib_{\cC}\]
	in \cite[Corollary 7.26]{Heine2024LocalglobalPrincipleParametrized}. Using \cref{proposition:Equivalence-DCorr-AF}, we see that
	\[\DProf_n \simeq \CondFib_{[n]} \simeq \Fun([n]^\hor,\mathrm{CORR}) \simeq \mathrm{CORR}_n\]
	naturally in $n$.		
\end{proof}

\section{Epilogue: Relation to other straightening results}\label{section:other-types-of-fibrations}

Given our main result \cref{theoremA} (cf.\ \cref{corollary:2nd-straightening-result-complete}), it is natural to ask how this straightening equivalence relates to various other known straightening equivalences. More generally, one may wonder whether Lurie's straightening-unstraightening for (locally) (co)cartesian fibrations and Ayala--Francis's straightening equivalence for Conduché fibrations follow from \cref{corollary:2nd-straightening-result-complete}. With some work, this indeed turns out to be the case. We will illustrate this in the case of Conduché fibrations and locally cocartesian fibrations, but similar arguments work for other types of fibrations too.

The first step is to unwind what the unital lax functor $P \colon \cC^\hor \to \DCorr$ corresponding to a functor $p \colon \cD \to \cC$ looks like. To this end, for a given $c$ in $\cC$ we will write $\cD_c$ for the fiber $\cD \times_{\cC} \{c\}$. Moreover, given $d,d'$ in $\cD$ and a map $f \colon p(d) \to p(d')$ in $\cC$ we will write $\Map_{/f}(d,d')$ for $\Map(d,d') \times_{\Map(p(d),p(d'))} \{f\}$. One obtains the following description:

\begin{enumerate}[(i)]
	\item\label{item1:unpacking} On objects, $P$ is given by $c \mapsto \cD_c$.
	\item\label{item2:unpacking} On horizontal arrows, $P$ sends $f \colon c \to d$ to the correspondence \[P(f) \colon \cD_c^\op \times \cD_d \to \Spc; \quad (x,y) \mapsto \Map_{/f}(x,y).\]
	\item\label{item3:unpacking} For any composable $c \xrightarrow{f} d \xrightarrow{g} e$, the lax comparison map $P(g)P(f) \to P(gf)$ is given by the natural map
	\begin{equation}\label{equation:lax-comparison-map-P}
		\int^{y \in \cD_d}\Map_{/f}(-,y) \times \Map_{/g}(y,-) \to \Map_{/gf}(-,-)
	\end{equation}
	induced by the composition maps
	\[\Map_{/f}(x,y) \times \Map_{/g}(y,z) \to \Map_{/gf}(x,z).\]
\end{enumerate}

The map \cref{equation:lax-comparison-map-P} is explicitly constructed in \cite[Lemma 2.2.7]{AyalaFrancis2020FibrationsInftyCategories}. Moreover, it is shown \cite[Lemma 2.2.8]{AyalaFrancis2020FibrationsInftyCategories} that $p \colon \cD \to \cC$ is a Conduché fibration precisely if for any composable $c \xrightarrow{f} d \xrightarrow{g} e$ in $\cC$, this lax comparison map \cref{equation:lax-comparison-map-P} is an equivalence. The equivalence $\Cat_{/\cC} \simeq \UnitLax(\cC^\hor,\DCorr)$ from \cref{corollary:2nd-straightening-result-complete} therefore restricts to the following equivalence:

\begin{theorem}[{\cite[Theorem 2.3.10]{AyalaFrancis2020FibrationsInftyCategories}}]\label{theorem:Ayala-Francis-straightening}
	There is a natural equivalence
	\begin{equation}\label{equation:Ayala-Francis-straightening}
		\CondFib_{\cC} \simeq \Fun(\cC^\hor,\DCorr)
	\end{equation}
	of categories, where $\CondFib_{\cC}$ denotes the category of Conduché fibrations over $\cC$. \qed
\end{theorem}

\begin{remark}
	Ayala--Francis work with the Segal space $\Corr$ instead of the double category $\DCorr$ and hence only obtain the equivalence \cref{equation:Ayala-Francis-straightening} at the level of groupoid cores. However, their proof strategy can easily be extended to obtain \cref{theorem:Ayala-Francis-straightening}.
\end{remark}

To recover Lurie's straightening equivalence for locally cocartesian fibrations, let us write $\Cat_{/\cC}^\mathrm{loccoc}$ for the full subcategory of $\Cat_{/\cC}$ spanned by the locally cocartesian fibrations.
Recall that $p \colon \cD \to \cC$ is a locally cocartesian fibration precisely if for every $f \colon c \to d$ in $\cC$ and $x$ in $\cD_c$, the functor $\Map_{/f}(x,-) \colon \cD_d \to \Spc$ is representable.
This condition can be entirely rephrased in terms of the double category $\DCorr$: it is equivalent to the horizontal morphism $P(f)$ being a \emph{companion} of a vertical morphism (see Definition 5.1 and Example 5.6 of \cite{Ruit2023FormalCategoryTheory}). Let us write $\DCorr^\mathrm{comp}$ for the locally full subcategory (in the sense of \cref{definition:locally-full-subcategory}) of $\DCorr$ spanned by the horizontal morphisms that are companions. It follows directly that $\Cat_{/\cC}^\mathrm{loccoc}$ is equivalent, under the equivalence of \cref{theoremA}, to $\UnitLax(\cC^\hor,\DCorr^\mathrm{comp})$.

To relate this to (oplax) functors into the 2-category $\Cat$, we will use the \emph{squares embedding} $\Sq \colon \Cat_2 \to \DblCat$, where $\Cat_2$ denotes the category of 2-categories (see \cref{example:2-categories}). For the definition of the squares embedding, see \cite[\S 10.4.2]{GaitsgoryRozenblyum2017StudyDerivedAlgebraica} or \cite[\S 3.6]{Ruit2024FunctorDoubleinftycategories}.

The squares embedding is related to the Gray tensor product in the following way. Given a category $\cC$, recall its horizontal and vertical embeddings $\cC^\hor$ and $\cC^\ver$ from \cref{example:Horizontal-embedding,example:Vertical-embedding}. It is shown in \cite[Remark 6.3]{Ruit2024FunctorDoubleinftycategories} that the Gray tensor product satisfies the following universal property: for any pair of categories\footnote{This universal property also holds when $\cC$ and $\cD$ are 2-categories, as was recently shown in \cite[Theorem B]{LoubatonRuit2025SquaresFunctorGaitsgoryRozenblyum}.} $\cC$ and $\cD$ and any 2-category $\cX$, there is a natural equivalence
\[\Map(\cC^\hor \times \cD^\ver, \Sq(\cX)) \simeq \Map(\cC \otimes \cD, \cX).\]
Using this universal property, it follows that for any 2-category $\cX$ there is a natural equivalence
\begin{equation}\label{equation:Oplax-natural-transformations-and-Sq}
	\UnitLax^\oplax(\cC,\cX) \simeq \UnitLax(\cC^\hor,\Sq(\cX))
\end{equation}
of categories. Here $\UnitLax^\oplax$ denotes the category of unital lax functors and oplax natural transformations between them.\footnote{To be precise, we define $\UnitLax^\oplax(\cC, \cX)$ as $\Fun^\oplax(\UnitEnv(\cC), \cX)$, where $\UnitEnv$ denotes the unital envelope of $\cC$ and $\Fun^\oplax$ is defined as in \cite[\nopp 10.3.2.7]{GaitsgoryRozenblyum2017StudyDerivedAlgebraica} or \cite[Definition 3.9]{Haugseng2021LaxTransformationsAdjunctions}.} One can then deduce the following:

\begin{theorem}[{\cite[\S 3]{Lurie2009InftyCategoriesGoodwillie}, \cite[Theorem B.57]{AyalaMazel-Geeea2024StratifiedNoncommutativeGeometry}}]
	There is a natural equivalence
	\[\Cat_{/\cC}^\mathrm{loccoc} \simeq \UnitOplax^\lax(\cC,\Cat)\]
	of categories, where $\Cat$ denotes the 2-category of categories. This equivalence restricts to an equivalence
	\[\LoccoCart(\cC) \simeq \UnitOplax(\cC,\Cat),\]
	where $\LoccoCart(\cC)$ is the category of locally cocartesian fibrations over $\cC$ and functors between them that preserve locally cocartesian morphisms.
\end{theorem}

\begin{proof}
	Let $\Ver(-)$ denote the vertical fragment functor defined in \cite[\S 3.3]{Ruit2024FunctorDoubleinftycategories}.
	It follows from \cref{proposition:Equivalence-DCorr-Ruit} and \cite[Definition 4.12 \& Proposition 4.19]{Ruit2023FormalCategoryTheory} that 
	$\Ver(\DCorr) \simeq \Ver(\DCorr^\mathrm{comp}) \simeq \Cat^\mathrm{co}$.
	Here $\Cat^\mathrm{co}$ denotes the 2-category of categories where the direction of the 2-morphisms is reversed.
	The double category $\DCorr^\mathrm{comp}$ is complete and accompanied in the sense of \cite[Definition 3.2.2]{Loubaton2025EffectivityGeneralizedDouble}, hence by Proposition 3.4.23 of op.\ cit.\ it is of the form $\Sq(\cX)$ for some 2-category $\cX$.
	It is shown in \cite[Proposition 3.58]{Ruit2024FunctorDoubleinftycategories} that $\Ver(\Sq(\cX)) \simeq \cX$, hence since $\Ver(\DCorr^\mathrm{comp}) \simeq \Cat^\mathrm{co}$, it follows that $\DCorr^\mathrm{comp} \simeq \Sq(\Cat^\mathrm{co})$.
	Combining this with \cref{equation:Oplax-natural-transformations-and-Sq}, it follows that
	\[\UnitLax(\cC^\hor,\DCorr^\mathrm{comp}) \simeq \UnitLax^\oplax(\cC,\Cat^\mathrm{co}) \simeq \UnitOplax^\lax(\cC,\Cat).\]
	By the discussion preceding this proof, the left-hand side is equivalent to the full subcategory $\Cat_{/\cC}^\mathrm{loccoc}$ of $\Cat_{/\cC}$ spanned by the locally cocartesian fibrations. Unwinding these equivalences, it follows that a map between locally cocartesian fibrations preserves locally cocartesian lifts precisely if the corresponding lax natural transformation in $\UnitOplax^\lax(\cC,\Cat)$ is a strict natural transformation.
	This completes the proof.
\end{proof}

%% file: part-II.tex
\part{The universal property of the Morita double category}\label{PartII}

In this second part we prove \cref{theoremB:Morita}. This result states that under mild assumptions on a double category $\DD$ one can construct its Morita double category $\Mor(\DD)$, which satisfies the universal property that for any double category $\DE$, there is a natural equivalence
\begin{equation}\label{equation:Universal-property-Mor-D}
	\Lax(\DE,\DD) \simeq \UnitLax(\DE,\Mor(\DD)).
\end{equation}

Roughly, the idea is as follows.
We first construct a ``lax'' version $\LaxMor(\DD)$ of the Morita double category.
This is a simplicial object in $\Cat$, though generally not a double category.
However, it has the universal property $\Lax(\DE,\DD) \simeq \Fun(\DE,\LaxMor(\DD))$, and we will define $\Mor(\DD)$ as a subobject of $\LaxMor(\DD)$.
We then reduce the proof of the natural equivalence \cref{equation:Universal-property-Mor-D} to the case where $\DE$ is the horizontal embedding $[n]^\hor$ of the category $[n]$.
In this particular case, we establish the equivalence by showing that there are natural fully faithful embeddings
\[\Lax([n]^\hor,\DD) \hookrightarrow \UnitLax([n]^\hor,\operatorname{\LaxMor} \DD) \text{ and } \UnitLax([n]^\hor,\operatorname{\Mor} \DD ) \hookrightarrow \UnitLax([n]^\hor,\operatorname{\LaxMor} \DD )\]
whose images agree.
This is achieved in \cref{the crucial lemma} by computing certain ``vertical Kan extensions'', an extension of the theory of operadic Kan extensions to double categories.

\begin{remark*}
	While we don't explicitly use them here, our proof strategy is heavily influenced by the theory of \emph{enhanced Segal objects} developed in \cite{Abellan2023ComparingLaxFunctors}.
\end{remark*}

\begin{notation*}
	Throughout this part, the composition of horizontal morphisms in a double category $\DD$ will be denoted by $\otimes$ and written from left to right; that is, given $X$ in $\DD(x,y)$ and $Y$ in $\DD(y,z)$, we write $X \otimes Y$ for their horizontal composite.
\end{notation*}

\section{(Unital) envelopes of double categories}\label{sec:Envelopes}

Given a category $\cB$ and a wide subcategory $\cB_0$, we will write $\coCart_0(\cB)$ for the subcategory of $\Cat_{/\cB}$ whose objects are those functors $\cC \to \cB$ that have all cocartesian lifts of maps in $\cB_0$, and whose morphisms are functors over $\cB$ that preserve these cocartesian lifts. Given two objects $\cC$ and $\cD$ in $\coCart_0(\cB)$, we will write $\Fun_{/\cB}^{\cB_0}(\cC,\cD)$ for the full subcategory of $\Fun_{/\cB}(\cC,\cD)$ spanned by those functors $\cC \to \cD$ over $\cB$ that preserve cocartesian lifts of maps in $\cB_0$.

Recall that the (non-full) inclusion
\[\Fun(\Delta^\op,\Cat) \simeq \coCart(\Delta^\op) \hookrightarrow \coCart_\inert(\Delta^\op)\]
admits a (2-categorical) left adjoint, see for example \cite[Corollary 2.2.5]{BarkanHaugsengea2022EnvelopesAlgebraicPatterns}.
We will denote this left adjoint by $\Env \colon \coCart_\inert(\Delta^\op) \to \Fun(\Delta^\op,\Cat)$ and call it the \emph{envelope} construction.
Moreover, for a double category $\DD$, we will simply write $\Env \DD$ for $\Env (\un \DD)$.
It is shown in \cite[Corollary 2.2.5]{BarkanHaugsengea2022EnvelopesAlgebraicPatterns} that $\Env$ sends $\cC \to \Delta^\op$ to (the straightening of) the pullback $\cC \times_{\Delta^\op} \Ar_\activ(\Delta^\op)$, where $\Ar_\activ(\Delta^\op)$ denotes the full subcategory of $\Ar(\Delta^\op)$ spanned by the active morphisms.
In particular, given a double category $\DD$, we see that
\[\Env\DD \colon \Delta^\op \to \Cat; \quad (\Env\DD)_n = \un \DD \times_{\Delta^\op} (\Delta^\activ_{[n]/})^\op.\]
Observe that $(\Env\DD)_1$ agrees with the wide subcategory $\un \DD_\activ$ of $\un \DD$ spanned by the morphisms that lie over active maps in $\Delta^\op$.
It is not hard to see that $\Env\DD$ is again a double category, hence $\Env$ restricts to give a left adjoint to the inclusion $\DblCat \hookrightarrow \DblCatlax$ and there are natural equivalences
\[\Lax(\DD,\DE) \simeq \Fun(\Env(\DD),\DE).\]

One can easily verify that the inclusion $\DblCat \hookrightarrow \coCart_\idle(\Delta^\op)$ is accessible and preserves limits, hence it must also admit a left adjoint. Since this inclusion lands in the full subcategory $\DblCatulax$ of $\coCart_{\idle}(\Delta^\op)$, we obtain a restricted adjunction
\[\begin{tikzcd}
	{\UnitEnv : \DblCatulax} & \DblCat
	\arrow[""{name=0, anchor=center, inner sep=0}, shift left=1.5, from=1-1, to=1-2]
	\arrow[""{name=1, anchor=center, inner sep=0}, shift left=1.5, hook', from=1-2, to=1-1]
	\arrow["\dashv"{anchor=center, rotate=-90}, draw=none, from=0, to=1]
\end{tikzcd}\]
whose left adjoint we will denote by $\UnitEnv$ and call the \emph{unital lax envelope}. Observe that since the inclusion $\DblCat \hookrightarrow \DblCatulax$ preserves cotensors by objects of $\Cat$, we obtain natural equivalences
\[\UnitLax(\DD,\DE) \simeq \Fun(\UnitEnv \DD, \DE)\]
of categories.
Finding an explicit description of $\UnitEnv \DD$ seems difficult for a general double category $\DD$.
However, for the double categories $[n]^\hor$ such a description exists:

\begin{lemma}[{An explicit description of $\UnitEnv [n]^\hor$}]\label{lemma:Unital-lax-envelope}
	The double category $\UnitEnv [n]^\hor$ is equivalent to the locally full subcategory $\DE$ of $\Env [n]^\hor$ spanned by those horizontal morphisms of $\Env [n]^\hor$ which correspond to injective maps $[k] \rightarrowtail [n]$ under the equivalence $(\Env [n]^\hor)_1 \simeq (\alln)_\activ$.
\end{lemma}

\begin{remark}
	Unwinding the definition of $\DE$, we see that $\DE_0$ is the set $\{0,\ldots,n\}$.
	The mapping category $\DE(k,l)$ is empty if $l < k$ and otherwise it is the poset of subsets of $\{k,\ldots,l\}$ containing $k$ and $l$, ordered by reverse inclusion.
	Composition is given by union of subsets.
	This is a 2-category commonly known as the (2-truncated) $n$-oriental $\mathcal{O}_n$ \cite{Street1987AlgebraOrientedSimplexes}, and \cref{lemma:Unital-lax-envelope} verifies that $\mathcal{O}_n$ classifies unital lax functors out of $[n]$ as defined in \cref{definition:lax-functors}.
\end{remark}

\begin{proof}
	Let us write $\iota \colon \DE \hookrightarrow \Env [n]^\hor$ for the locally full subcategory described in this lemma.
	Observe that both $\DE$ and $\Env [n]^\hor$ are 2-categories in the sense of \cref{example:2-categories}.
	We first verify that $\iota$ admits a right adjoint $R \colon \Env [n]^\hor \to \DE$, in the sense that there exist a unit morphism $\id \to R \iota$ in $\Fun(\DE,\DE)$ and a counit morphism $\iota R \to \id$ in $\Fun(\Env [n]^\hor, \Env [n]^\hor)$ satisfying the triangle identities.\footnote{
	Here we use the notation $\Fun(-,-)$ as defined at the end of \cref{definition:lax-functors}.
	This is \textbf{not} the usual category of functors between 2-categories, since the 1-morphisms in $\Fun(-,-)$ are not natural transformations.
	Instead, they are higher-categorical analogues of the ``icons'' defined in \cite{Lack2010Icons}.}
	Observe that $\iota_0 \colon \DE_0 \to (\Env [n]^\hor)_0$ is a bijection of sets and that $\iota_1 \colon \DE_1 \to (\Env [n]^\hor)_1$ has a right adjoint $R_1$ given by factoring a map $[k] \to [n]$ through its image.
	Combining Proposition A.3.16 and Lemma A.3.13 of \cite{BlansBlom2024ChainRuleGoodwillie}, we obtain a functor $R \colon \Env [n]^\hor \to \DE$ right adjoint to $\iota$ in the sense described above.
	Moreover, since $\iota_1 \colon \DE_1 \to \Env [n]^\hor$ is fully faithful, the unit $\id \to R \iota$ is an equivalence.

	Applying $\Fun(-,\DD)$, one obtains an adjunction
	\[R^* :  \Fun(\DE,\DD) \rightleftarrows \Fun(\Env [n]^\hor, \DD) : \iota^*\]
	whose left adjoint $R^*$ is fully faithful.
	It is easily verified that for any $F \colon \Env [n]^\hor \to \DD$, the counit $R^* \iota^* F \to F$ is invertible if and only if, under the equivalence $\Fun(\Env [n]^\hor, \DD) \simeq \Lax([n]^\hor, \DD)$, the functor $F$ corresponds to a unital lax functor. This shows that there is a natural equivalence $\Fun(\DE,\DD) \simeq \UnitLax([n]^\hor,\DD)$ given by restricting along the inclusion $[n]^\hor \hookrightarrow \DE$.
\end{proof}

\section{Cofree cocartesian fibrations}

The inclusions of $\DblCat$ into $\DblCatulax$ and $\DblCatlax$ don't preserve all colimits, so they can't admit right adjoints. However, it turns out that the inclusions of $\coCart(\Delta^\op)$ into $\coCart_\idle(\Delta^\op)$ and $\coCart_\inert(\Delta^\op)$ do admit right adjoints. More generally, one can show the following.

\begin{proposition}[Partially cofree cocartesian fibrations]\label{proposition:Partially-cofree-fibrations}
	Let $\cB$ be a category and $\cB_0$ a wide subcategory. Then the inclusion
	\[\Fun(\cB,\Cat) \simeq \coCart(\cB) \hookrightarrow \coCart_{0}(\cB)\]
	admits a right adjoint $R_0$, which is explicitly given by
	\[R_0(X) \colon \cB \to \Cat; \quad R_0(X)(b) = \Fun_{/\cB}^{\cB_0}(\cB_{b/},X).\]
	Moreover, there are natural equivalences
	\begin{equation}\label{equation:cofree-cocartesian-fibration-2-categorical}
		\Fun(Y, R_0(X)) \simeq \Fun_{/\cB}^{\cB_0}(\un Y, X).
	\end{equation}
\end{proposition}

\begin{proof}
	We first show that $\coCart(\cB) \hookrightarrow \coCart_0(\cB)$ admits a right adjoint when viewed as a functor between categories (as opposed to 2-categories). Since $\coCart(\cB)$ is presentable, by \cite[Corollary 5.5.2.9]{HTT} it suffices to show that $\coCart(\cB) \hookrightarrow \coCart_0(\cB)$ preserves colimits. Let a diagram $f' \colon I \to \coCart(\cB)$ be given and write $f$ for the composite $I \to \coCart(\cB) \to \Cat$, where the second functor sends a cocartesian fibration to its source. It follows as in \cite[Proposition A.1]{Ramzi2022MonoidalGrothendieckConstruction} that one can compute the colimit of $f$ as $\un f[W^{-1}]$, and that this is also the colimit of $f'$ in $\coCart(\cB)$ and $\Cat_{/\cB}$. Here $W$ is the collection of cocartesian edges in $\un f$. Now suppose we are given a  cocone $f' \Rightarrow \cD$ in $\coCart_0(\cB)$. The universal property of the colimit in $\Cat_{/\cB}$ provides us with a functor $P \colon \un f[W^{-1}] \to \cD$ over $\cB$. To see that $\un f[W^{-1}]$ is the colimit of $f'$ in $\coCart_0(\cB)$, it suffices to show that $P$ preserves cocartesian lifts of maps in $\cB_0$. This follows exactly as at the end of the proof of \cite[Proposition A.1]{Ramzi2022MonoidalGrothendieckConstruction}.
	
	To obtain the equivalence of functor categories \cref{equation:cofree-cocartesian-fibration-2-categorical}, it suffices to show that $\coCart(\cB) \hookrightarrow \coCart_0(\cB)$ preserves tensors by objects of $\Cat$. This is clear since in both cases tensors are computed as cartesian products.
	
	Finally, the explicit description of $R_0(X)$ as a functor $\cB \to \Cat$ follows from the natural equivalences
	\[R_0(X)(b) \simeq \Fun_{/\cB}^\mathrm{cocart}(\cB_{b/},\un R_0(X)) \simeq \Fun_{/\cB}^{\cB_0}(\cB_{b/},X). \qedhere\]
\end{proof}

\begin{definition}
	The right adjoints of $\Fun(\Delta^\op, \Cat) \hookrightarrow \coCart_\idle(\Delta^\op)$ and $\Fun(\Delta^\op, \Cat) \hookrightarrow \coCart_\inert(\Delta^\op)$ will be denoted by $R_\idle$ and $R_\inert$, respectively.
\end{definition}

\section{The lax Morita double category}

As a first step towards constructing the Morita double category of a double category $\DD$, we will construct a slightly larger object.

\begin{definition}[Lax Morita double category]\label{definition:lax-morita-category}
	Let $\DD$ be a double category. Its \emph{lax Morita double category} $\LaxMor(\DD)$ is the functor
	\[\LaxMor(\DD)_\bullet \colon \Delta^\op \to \Cat; \quad \LaxMor(\DD)_n = \Lax([n]^\hor,\DD).\]
\end{definition}

\begin{warning}
	Unlike what the name suggests, the lax Morita double category $\LaxMor(\DD)$ generally does \emph{not} satisfy the Segal condition and is hence not a double category.
\end{warning}

Note that $\alln \to \Delta^\op$ is the cocartesian unstraightening of $[n]^\hor$. In particular, $\LaxMor(\DD)$ is equivalent to $R_\inert(\un \DD)$ and we obtain the following.

\begin{proposition}[{Universal property of $\LaxMor(\DD)$}]\label{proposition:Universal-property-LaxMor}
	Let $\DD$ and $\DE$ be double categories. Then there is an equivalence
	\[\Lax(\DE,\DD) \simeq \Fun(\DE,\LaxMor(\DD))\]
	that is natural in $\DD$ and $\DE$.
\end{proposition}

\begin{proof}
	This follows directly from \cref{proposition:Partially-cofree-fibrations}.
\end{proof}

Since $\LaxMor(\DD)_0$ is the category $\Lax([0]^\hor,\DD) = \Fun_{/\Delta^\op}^\inert(\Delta^\op,\DD)$, an object of $\LaxMor(\DD)_0$ consists of an object $d$ of $\DD$ together with an algebra in $\DD(d,d)$ with respect to the horizontal composition monoidal structure. Similarly, an object of $\LaxMor(\DD)_1$ describes algebras $A$ in $\DD(d,d)$ and $B$ in $\DD(e,e)$ together with an $(A,B)$-bimodule $M$ in $\DD(d,e)$. The idea of the Morita double category $\Mor(\DD)$ of $\DD$ is to define horizontal composition by the relative tensor product of such bimodules.
Concretely, we will define $\Mor(\DD)$ as a certain subobject of $\LaxMor(\DD)$ such that the objects in $\Mor(\DD)_2$ are ``witnesses'' for the fact that composition is given by the relative tensor product.
To make this precise, we will need the following assumptions on $\DD$.

\begin{definition}[Moritable double category]\label{def:Moritable}
	Let $\DD$ be a double category. We call $\DD$ \emph{Moritable} if 
	\begin{enumerate}[(a)]
		\item\label{item1:Moritable} for any two objects $x, y \in \DD_0$, the category $\DD(x,y)$ admits geometric realizations,
		\item\label{item2:Moritable} for any three objects $x,y,z \in \DD_0$, the composition functor
		\[\DD(x,y) \times \DD(y,z) \to \DD(x,z)\]
		preserves geometric realizations in each variable separately, and
		\item\label{item:Framed-double-category} the source-target projection $\DD_1 \to \DD_0 \times \DD_0$ is a bicartesian fibration.
	\end{enumerate}
\end{definition}

\begin{remark}
	Items \ref{item1:Moritable} and \ref{item2:Moritable} should not be surprising: these ensure that the relative tensor product exists and that it is associative, which is certainly necessary to construct a Morita (double) category. Item \ref{item:Framed-double-category} might seem technical and be a bit more surprising.
	This property is used in the proof of \cref{lem:VLan-existence}, because it allows one to reduce the computation of certain relative left Kan extensions to computations of certain colimits in mapping categories.
\end{remark}

\begin{remark}[Framed double category]\label{remark:Framed-double-category}
	A double category satisfying item \ref{item:Framed-double-category} of \cref{def:Moritable} is usually called a \emph{framed double category}. Using the Segal condition and induction, this condition implies that $\DD_n \to \DD_0^{\times n+1}$ is a bicartesian fibration for every $n \geq 1$.
\end{remark}

\begin{example}
	Let $\cC$ be a monoidal category, viewed as a double category via \cref{example:Moniodal-as-double}. Then item \ref{item:Framed-double-category} of \cref{def:Moritable} is automatic, hence $\cC$ is Moritable precisely if its underlying category admits geometric realizations and these are preserved by the monoidal structure in each variable separately.
\end{example}

\section{Vertical Kan extensions}\label{sec:vertical_kan_extensions}

To construct the Morita double category in \cref{sec:constructing_the_morita_double_category} and show that it satisfies the universal property of \cref{theoremB:Morita}, we will use a theory of ``vertical Kan extensions''.
This is an extension of the theory of operadic left Kan extensions from \cite[\S 3.1.2]{HA} to double categories.
We will only require a fraction of this theory, namely when all colimits involved are sifted.
Since this fraction is simpler than the general theory of vertical left Kan extensions, we will only focus on this case.
Recall from \cref{sec:Envelopes} that we write $\Env$ for the left adjoint to the inclusion
\[\Fun(\Delta^\op,\Cat) \simeq \coCart(\Delta^\op) \hookrightarrow \coCart_\inert(\Delta^\op).\]

\begin{definition}[Vertical left Kan extensions]\label{def:VertKan}
	Let $f \colon \cC \to \cD$ be a functor in $\coCart_{\inert}(\Delta^\op)$ and let $\DD$ be a double category.
	We say that $\DD$ \emph{admits vertical left Kan extensions along $f$} if the restriction functor
	\[\Fun_{/\Delta^\op}^\mathrm{inert}(\cD,\DD) \simeq \Fun(\Env \cD,\DD) \xrightarrow{f^*} \Fun(\Env \cC,\DD) \simeq \Fun_{/\Delta^\op}^\mathrm{inert}(\cC,\DD)\]
	admits a left adjoint $\VLan_f$.
	In this situation, given $g$ in $\Lax(\cC, \DD)$, we call $\VLan_f(g)$ the \emph{vertical left Kan extension} of $g$ along $f$.
\end{definition}

Our main result in this section is an existence result for such vertical left Kan extension along maps of \emph{generalized planar $\infty$-operads}.
These are defined in \cite[Definition 3.1.13]{GepnerHaugseng2015EnrichedCategoriesNonsymmetric} as a non-symmetric variant of Lurie's generalized $\infty$-operads \cite[\S 2.3.2]{HA}.
Instead of recalling their usual definition, we simply observe that by \cite[Observation 4.1.3 \& Example 4.1.9]{BarkanHaugsengea2022EnvelopesAlgebraicPatterns}, an object $\cP \to \Delta^\op$ of $\coCart_\inert(\Delta^\op)$ is a generalized planar $\infty$-operad if and only if $\Env(\cP) \colon \Delta^\op \to \Cat$ is a double category.
Given a generalized planar $\infty$-operad $\cP \to \Delta^\op$, we write $\cP_n$ for its fiber over $[n]$.

We require one more definition.

\begin{definition}[Primitive horizontal arrows]\label{def:primite-arrow}
	Let $\cP$ be a generalized planar $\infty$-operad. We call a horizontal arrow in $\Env \cP$ \emph{primitive} if it lies in the image of the inclusion $\cP_1 \hookrightarrow (\Env \cP)_1 \simeq \cP_{\activ}$.
\end{definition}

Observe that any horizontal arrow in $\Env \cP$ can be written as a composite of primitive horizontal arrows.

The main result of this section is the following existence criterion for vertical left Kan extension. Its proof is quite long and technical, and the reader should feel free to treat it as a black box.

\begin{lemma}\label{lem:VLan-existence}
	Let $\DD$ be a Moritable double category and
	\[\begin{tikzcd}[column sep=tiny]
		\cP && \cQ \\
		& {\Delta^\op}
		\arrow["f", from=1-1, to=1-3]
		\arrow[from=1-1, to=2-2]
		\arrow[from=1-3, to=2-2]
	\end{tikzcd}\]
	a map of generalized planar $\infty$-operads.
	Suppose that
	\begin{enumerate}[(1)]
		\item\label{VLan:cond1} $f \colon \cP \to \cQ$ is fully faithful,
		\item\label{VLan:cond2} $\cP_0$ and $\cQ_0$ are discrete sets and $f_0 \colon \cP_0 \to \cQ_0$ is a bijection, and
		\item\label{VLan:cond3} for any $x,y$ in $\cQ_0$ and any primitive $H$ in $(\Env \cQ)(x,y)$, the slice $(\Env \cP)(x,y)_{/H}$ admits a right cofinal\footnote{What we call \emph{right cofinal} is called \emph{cofinal} in \cite{HTT}. The terminology used here agrees with \cite{kerodon}.} functor from $\Delta^\op$.
	\end{enumerate}
	Then $\DD$ admits vertical left Kan extensions along $f$. Moreover, the functor
	\[\VLan_f \colon \Fun(\Env \cP,\DD) \to \Fun(\Env \cQ, \DD)\]
	is fully faithful, and its essential image is spanned by those double functors $G \colon \Env \cQ \to \DD$ with the property that for any $x,y$ in $\cQ_0$, the functor $G_{x,y} \colon (\Env \cQ)(x,y) \to \DD(G(x),G(y))$ is left Kan extended along $f_{x,y} \colon (\Env \cP)(x,y) \hookrightarrow (\Env \cQ)(x,y)$.
\end{lemma}

\begin{proof}
	Since $f \colon \cP \to \cQ$ is fully faithful, this also holds for $\un \Env \cP \to \un \Env \cQ$. We will show the following:

    \begin{claim}
	For any square of the form
	\begin{equation}\label{eq:Relative-LKan-square}
		\begin{tikzcd}
			\un \Env \cP \ar[r,"F"] \ar[d,"f"',hook] & \un \DD \ar[d,"p"]\\
			\un \Env \cQ \ar[r," "] \ar[ur,dashed,"\Lan_f F"' {xshift=-8,yshift=-2}] & \Delta^\op
		\end{tikzcd}
	\end{equation}
	such that $F$ preserves cocartesian morphisms, there exists a $p$-left Kan extension $\Lan_f F$ in the sense of \cite[Definition 4.3.2.2]{HTT}, which moreover preserves cocartesian morphisms.
	\end{claim}

	Combined with \cite[Proposition 4.3.2.17]{HTT}\footnote{Strictly speaking, the assumptions of this proposition might not be satisfied, since the claim does not guarantee that every functor $\un \Env \cP \to \un \DD$ over $\Delta^\op$ admits a $p$-left Kan extension. However, the proof of \cite[Proposition 4.3.2.17]{HTT} still goes through and shows that the left adjoint $\VLan_f$ exists.}, this claim guarantees the existence of the desired fully faithful left adjoint
	\[\VLan_f \colon \Lax(\cP,\DD) \simeq \Fun_{/\Delta^\op}^\mathrm{cocart}(\un \Env \cP, \un \DD) \hookrightarrow \Fun_{/\Delta^\op}^\mathrm{cocart}(\un \Env \cQ, \un \DD) \simeq \Lax(\cQ,\DD).\]

	We will now prove the claim, and afterwards characterize the essential image of this left adjoint.

	\textbf{Step 1. Existence of $\Lan_f F$.} To see that $\Lan_f F$ exists, it suffices by \cite[Lemma 4.3.2.13]{HTT} to show that the composite
	\[(\un \Env \cP)_{/\vec{x}} \to \un \Env \cP \xrightarrow{F} \un \DD\]
	admits a $p$-colimit for any $m \geq 0$ and any object $\vec{x} = (x_1,\ldots,x_m)$ in $(\un \Env \cQ)_m$. Since $\un \Env \cP \to \un \Env \cQ$ preserves all cocartesian morphisms over $\Delta^\op$, the inclusion $((\Env \cP)_m)_{/\vec{x}} \hookrightarrow (\un \Env \cP)_{/\vec{x}}$ is right cofinal.
	By \cite[Proposition 4.3.1.8]{HTT}, it suffices to show that the restriction
	\[F' \colon ((\Env \cP)_m)_{/\vec{x}} \to \un \DD\]
	admits a $p$-colimit. Since this diagram lands in the fiber $\DD_m$, we will first show that it admits a colimit there. For $m=0$ this is clear, so suppose $m > 0$. Write $i_{j-1}$ and $i_j$ for the source and target of $x_j$ and write $d_j$ for $F(i_j)$. Consider the commutative square
	\[\begin{tikzcd}[column sep=60]
		((\Env \cP)_m)_{/\vec x} \ar[r,"F'"] \ar[d,""] & \DD_m \ar[d,"q"]\\
		((\Env \cP)_m)_{/\vec x}^\triangleright \ar[r,"\const_{(d_0,\ldots,d_m)}"] & \DD_0^{\times m+1},
	\end{tikzcd}\]
	where the bottom horizontal functor is the constant functor equal to $(d_0,\ldots, d_m)$. The functor $q$ is bicartesian by \cref{remark:Framed-double-category}. By Propositions 4.3.1.5(2) and 4.3.1.10 of \cite{HTT}, it follows that the $q$-colimit of $F'$ can be computed as an ordinary colimit in the fiber of $q$; that is, it is computed as the colimit of
	\begin{equation}\label{eq:colim-mappingcats}
		((\Env \cP)_m)_{/\vec{x}} \to \DD(d_0,d_1) \times \cdots \times \DD(d_{m-1},d_m).
	\end{equation}
	Now observe that $((\Env \cP)_m)_{/\vec{x}} \simeq (\Env \cP)(i_0,i_1)_{/x_1} \times \cdots \times (\Env \cP)(i_{m-1},i_m)_{/x_m}$ since $\cP_0$ is discrete.
	It follows from \cref{lem:primi1} that the colimit of \cref{eq:colim-mappingcats} can be computed as a colimit of the form
	\begin{equation}\label{eq:colim-sifted-version-mappingcats}
		\Delta^\op \times \cdots \times \Delta^\op \to \DD(d_0,d_1) \times \cdots \times \DD(d_{m-1},d_m).
	\end{equation}
	Since $\Delta^\op$ is sifted, this colimit is componentwise computed as a geometric realization, which exists since $\DD$ is Moritable.
	To see that this defines a $p$-colimit for the diagram \cref{eq:Relative-LKan-square}, we need to show by \cite[Proposition 4.3.1.10]{HTT} that for any $\tau \colon [k] \to [m]$, the functor $\tau^* \colon \DD_m \to \DD_k$ preserves this colimit.
	Using that $\DD_k \to \DD_0^{k+1}$ is a bicartesian fibration, this reduces to showing that 
	\[\DD(d_0,d_1) \times \cdots \times \DD(d_{m-1},d_m) \to \DD(d_{\tau(0)},d_{\tau(1)}) \times \cdots \times \DD(d_{\tau(k-1)}, d_{\tau(k)})\]
	preserves the colimit of \cref{eq:colim-sifted-version-mappingcats}. If $\tau$ is idle this is clear, so it suffices to show this when $\tau$ is an inner face map.
	This case follows since the composition maps $\DD(c,d) \times \DD(d,e) \to \DD(c,e)$ preserve geometric realizations.
	We therefore conclude that the desired $p$-colimit exists and hence by \cite[Lemma 4.3.2.13]{HTT} that the relative left Kan extension \cref{eq:Relative-LKan-square} exists.
	
	\textbf{Step 2. $\Lan_f F$ preserves cocartesian morphisms.}
	Given the computations above, showing that $\Lan_f F$ preserves cocartesian morphisms over $\Delta^\op$ amounts to showing that for any $\tau \colon [n] \to [m]$ in $\Delta^\op$ and any $\vec{x}$ in $(\Env \cP)_m$, the square
	\[\begin{tikzcd}
		{((\Env \cP)_m)_{/\vec{x}}} & {\DD_m} \\
		{((\Env \cP)_n)_{/\tau^! \vec{x}}} & {\DD_n}
		\arrow[from=1-1, to=1-2]
		\arrow["{{\tau}^!}"', from=1-1, to=2-1]
		\arrow["{\tau^*}", from=1-2, to=2-2]
		\arrow[from=2-1, to=2-2]
	\end{tikzcd}\]
	induces an equivalence
	\[\colim_{y \in (\Env \cP)_{/\vec{x}}} \tau^* F(y) \to \colim_{y \in (\Env \cP)_{/\tau^! \vec{x}}} F(y).\]
	Here $\tau^!$ denotes the cocartesian transport functor along $\tau$.
	It therefore suffices to show that $\tau^!$ is right cofinal. If $\tau$ is an outer face map, this follows since $\Delta^\op$ is weakly contractible. If $\tau$ is an inner face map, then this follows since for any $H$ in $(\Env \cQ)(x,y)$ and $H'$ in $(\Env \cQ)(y,z)$, the composition map
	\[\otimes \colon (\Env \cP)(x,y)_{/H} \times (\Env \cP)(y,z)_{/H'} \to (\Env \cP)(x,z)_{/H \otimes H'}\]
	is an equivalence of categories by \cref{lem:tensordisjunctive}.
	Finally, when $\tau$ is a degeneracy map, this follows since the inclusion $\{ \unit_x \} \to (\Env \cP)(x,x)_{/\unit_x}$ is the inclusion of a terminal object.
	
	This concludes the proof of the claim, hence the fully faithful left adjoint $\VLan_f$ exists.

	\textbf{Step 3. The essential image of $\VLan_f$.}
	Finally, we characterize the essential image of $\VLan_f \colon \Fun(\Env \cP, \DD) \hookrightarrow \Fun(\Env \cQ, \DD)$. 
	Observe that the computation of the relative right Kan extension shows that for any $F \colon \Env \cP \to \DD$, the double functor $\VLan_f F \colon \Env \cQ \to \DD$ is of the desired form.
	Conversely, let $G \colon \Env \cQ \to \DD$ be given with the property that $G_{x,y} \colon (\Env \cQ)(x,y) \to \DD(G(x),G(y))$ is left Kan extended from $(\Env \cP)(x,y)$, and consider the counit $\varepsilon: \VLan_f(G|_{\Env \cP}) \Rightarrow G$.
	Note that $\varepsilon_1 \colon \VLan_f(G|_{\Env \cP})_1 \Rightarrow G_1$ is by construction an equivalence when restricted to $\Env \cP$.
	Since $G_{x,y} \colon (\Env \cQ)(x,y) \to \DD(G(x),G(y))$ is left Kan extended from $(\Env \cP)(x,y)$, it follows that $\varepsilon_1$ is an equivalence.
	We conclude from \cref{lem:double-functors-equivalence} that $\varepsilon \colon \VLan_f(G|_{\Env \cP})\Rightarrow G$ is an equivalence and hence that $G$ lies in the essential image of $\VLan_f$.
\end{proof}

We used the following lemmas.

\begin{lemma}\label{lem:tensordisjunctive}
	Let $\cP$ be a generalized planar $\infty$-operad such that $\cP_0$ is a set, and let $x \xrightarrow{H} y \xrightarrow{H'} z$ be a pair of composable horizontal arrows in $\Env \cP$.
	Then the composition map
	\[\otimes \colon (\Env \cP)(x,y)_{/H} \times (\Env \cP)(y,z)_{/H'} \to (\Env \cP)(x,z)_{/H \otimes H'}\]
	is an equivalence.
\end{lemma}

\begin{proof}
	Suppose $H$ and $H'$ lie in the fibers of $\cP \to \Delta^\op$ over $[m]$ and $[n]$, respectively.
	When slicing the horizontal composition
	\[(\Env \cP)_2 \simeq \cP \times_{\Delta^\op} (\Delta_{[2]/}^\activ)^\op \xrightarrow{d_1} \cP \times_{\Delta^\op} (\Delta_{[1]/}^\activ)^\op \simeq (\Env \cP)_1\]
	over $H \otimes H'$, one obtains the identity
	$\cP^\activ_{/H \otimes H'} \to \cP^\activ_{/H \otimes H'}$.
	Since $\Env \cP$ satisfies the Segal condition and $\cP_0$ is discrete, this is map is equivalent to the composition map
	\[\otimes \colon (\Env \cP)(x,y)_{/H} \times (\Env \cP)(y,z)_{/H'} \to (\Env \cP)(x,z)_{/H \otimes H'}. \qedhere\]
\end{proof}

\begin{lemma}\label{lem:primi1}
	Let $f \colon \cP \to \cQ$ be a fully faithful map of generalized planar $\infty$-operads such that $f_0 \colon \cP_0 \to \cQ_0$ is a bijection of sets, and suppose that for any primitive horizontal arrow $H \colon x \to y$ in $\Env \cQ$, the slice $(\Env \cP)(x,y)_{/H}$ admits a right cofinal functor from $\Delta^\op$.
	Then this holds for every horizontal arrow $H$ in $\Env \cQ$.
\end{lemma}

\begin{proof}
	Given a horizontal arrow $H$ in $\Env \cQ$, we may write it as a horizontal composite $H_1 \otimes \cdots \otimes H_n$ of primitive horizontal arrows.
	By \cref{lem:tensordisjunctive}, we obtain an equivalence
	\[(\Env \cP)(x_0,x_n)_{/H} \simeq (\Env \cP)(x_0,x_1)_{/H_1} \times \cdots \times (\Env \cP)(x_{n-1},x_n)_{/H_n}.\]
	Since each factor receives a right cofinal functor from $\Delta^\op$ and $\Delta^\op$ is sifted, it follows that $(\Env \cP)(x_0,x_n)_{/H}$ receives a right cofinal functor from $\Delta^\op$.
\end{proof}

The essential image of $\VLan_f$ in \cref{lem:VLan-existence} admits a slightly simpler description which is worth mentioning.

\begin{lemma}\label{lem:primi2}
	Let $\DD$ and $f \colon \cP \to \cQ$ be as in \cref{lem:VLan-existence}. Then $G \colon \Env \cQ \to \DD$ lies in the image of $\VLan_f$ if and only if for any primitive horizontal arrow $H \colon x \to y$ in $\Env \cQ$, the diagram
	\[((\Env \cP)(x,y)_{/H})^\triangleright \to (\Env \cQ)(x,y) \xrightarrow{G_{x,y}} \DD(G(x),G(y))\]
	is a colimit diagram.
\end{lemma}

\begin{proof}
	The ``only if'' direction is clear, so let $H$ be a horizontal arrow in $\Env \cQ$. As in \cref{lem:primi1}, we obtain an equivalence
	\[(\Env \cP)(x_0,x_n)_{/H} \simeq (\Env \cP)(x_0,x_1)_{/H_1} \times \cdots \times (\Env \cP)(x_{n-1},x_n)_{/H_n}\]
	where each $H_i$ is primitive.
	Using that every factor admits a right cofinal map from $\Delta^\op$, that $\Delta^\op$ is sifted, and that horizontal composition in $\DD$ preserves geometric realizations, we see that
	\[((\Env \cP)(x_0,x_n)_{/H})^\triangleright \to (\Env \cQ)(x_0,x_n) \xrightarrow{G} \DD(G(x_0),G(x_n))\]
	is a colimit diagram if for every $1 \leq i \leq n$,
	\[((\Env \cP)(x_{i-1},x_i)_{/H_i})^\triangleright \to (\Env \cQ)(x_{i-1},x_i) \xrightarrow{G} \DD(G(x_{i-1}),G(x_i))\]
	is a colimit diagram.
	This holds by assumption.
\end{proof}

\section{Constructing the Morita double category}\label{sec:constructing_the_morita_double_category}

An object $\sigma$ of $\LaxMor(\DD)_n = \Lax([n]^\hor,\DD)$ can be thought of as a collection of algebras $A_0,\ldots,A_n$ in $\DD$ together with $(A_i,A_j)$-bimodules $M_{ij}$ for all $i < j$ and suitably coherent bilinear maps $M_{ij} \otimes M_{jk} \to M_{ik}$ for all $i < j < k$.
This allows us to define relative tensor products.

\begin{construction}\label{const:Qijk}
For every $0 \leq i < j < k \leq n$, we define a functor $Q_{ijk} \colon (\Delta^\op)^\triangleright \to (\Env [n]^\hor)_1 = (\alln)_\activ$ by sending a linearly ordered set $[m]$ to the map
\[Q_{ijk}([m]) \colon [m+2] \to [n];\quad t \mapsto \begin{cases}
	i \quad &\text{if} \quad t = 0 \\
	j \quad &\text{if} \quad 0 < t < m+2 \\
	k \quad &\text{if} \quad t = m+2.
\end{cases} \]
Here $[-1]$ should be interpreted as the cone point. If we view $\sigma \in \LaxMor(\DD)_n$ as a functor $\Env [n]^\hor \to \DD$, then precomposition with $Q_{ijk}$ yields the diagram
\[\sigma \circ Q_{ijk} \colon (\Delta^\op)^\triangleright \to \DD(\sigma(i),\sigma(k)); \quad [m] \mapsto \begin{cases}
	M_{ij} \otimes A_j^{\otimes m} \otimes M_{jk} \quad &\text{if} \quad m \geq 0\\
	M_{ik} \quad &\text{if} \quad m = -1.
\end{cases}\]
\end{construction}

If this is a colimit diagram, then we say that the maps $M_{ij} \otimes A_j^{\otimes \bullet} \otimes M_{jk} \to M_{ik}$ \emph{exhibit $M_{ik}$ as the relative tensor product} $M_{ij} \otimes_{A_j} M_{jk}$.

\begin{definition}[Composite objects]\label{definition:composite-objects}
	An object $\sigma$ of $\LaxMor(\DD)_n$ as above is called \emph{composite} if for every $0 \leq i < j < k \leq n$, the maps $M_{ij} \otimes A_j^{\otimes \bullet} \otimes M_{jk} \to M_{ik}$ exhibit $M_{ik}$ as the relative tensor product $M_{ij} \otimes_{A_j} M_{jk}$.
\end{definition}

\begin{remark}\label{remark:everything-composite-degree-01}
	This condition is vacuous if $n=0,1$.
\end{remark}

One can deduce directly from the definition of composite objects that they form a subobject of $\LaxMor(\DD)$.

\begin{proposition}\label{prop:composite-objects-preserved-by-structure-maps}
	Let $\DD$ be a Moritable double category. Then for any map $\tau \colon [m] \to [n]$, the map $\tau^* \colon \LaxMor(\DD)_n \to \LaxMor(\DD)_m$ preserves composite objects.
\end{proposition}

\begin{proof}
	We first show this when $\tau$ is a face map $\delta_l \colon [n-1] \to [n]$.
	Let $\alpha \colon \Env [n]^\hor \to \DD$ be a composite object in $\LaxMor(\DD)_n$
	and let $Q_{ijk} \colon (\Delta^\op)^\triangleright \to (\Env [n-1]^\hor)_1$ be as in \cref{const:Qijk}.
	Then $\tau \circ Q_{ijk} \colon (\Delta^\op)^\triangleright \to (\Env [n]^\hor)_1$ is again a functor of the form $Q_{ijk}$ (potentially for different $i < j < k$), hence $\tau^*\alpha = \alpha \circ \tau \colon \Env [n-1]^\hor \to \DD$ is composite.

	Now suppose $\tau$ is a degeneracy map $\sigma_i \colon [n+1] \to [n]$.
	Using a similar argument as above, to see that $\sigma^* \alpha \colon \Env [n+1]^\hor \to \DD$ is composite, it suffices to show that for any $j < i$ and any $k > i$, the maps
	\[M_{j,i} \otimes_{A_i} A_i \to M_{j,i} \quad \text{and} \quad A_i \otimes_{A_i} M_{i,k} \to M_{i,k}\]
	are equivalences.
	This follows by the usual extra degeneracy argument (cf.\ Example 4.7.2.7 of \cite{HA}).
\end{proof}

\begin{definition}[Morita double category]
	Let $\DD$ be a Moritable double category. Its \emph{Morita double category} $\Mor(\DD)$ is defined by letting $\Mor(\DD)_n$ be the full subcategory of $\LaxMor(\DD)_n$ spanned by the composite objects.
\end{definition}

Observe that by \cref{prop:composite-objects-preserved-by-structure-maps}, $\Mor(\DD)$ is indeed a well-defined object in $\Fun(\Delta^\op,\Cat)$.
In what follows, we will use the following variant of \cref{const:Qijk}.

\begin{construction}\label{const:barQij}
	Let $0 \leq i < j \leq n$.
	We will write $\star$ for the concatenation of two linearly ordered sets.
	We define the concatenation functor
	\[\overline{Q}_{ij} \colon (\Delta^{\times j-i-1,\op})^\triangleright \to \Delta^\op\]
	by
	\[\overline{Q}_{ij}(I_{i+1},\ldots,I_{j-1}) = [0] \star I_{i+1} \star \cdots \star I_{j-1} \star [0],\]
	where we send the cone point to $[0] \star [0]$.
	Note that $\overline{Q}_{ij}(I_{i+1},\ldots,I_{j-1})$ comes with a canonical map to $[n]$, given by
	\[\overline{Q}_{ij}(I_{i+1},\ldots,I_{j-1}) \to [n]; \quad t \mapsto
	\begin{cases}
		i \quad &\text{if} \quad t \text{ is minimal} \\
		k \quad &\text{if} \quad t \text{ lies in } I_k \\
		j \quad &\text{if} \quad t \text{ is maximal}.
	\end{cases} \]
	We can therefore see $\overline{Q}_{ij}$ as a functor landing in $(\alln)_{\activ} = (\Env [n]^\hor)_1$.
	Observe that it even lands in the subcategory $(\Env [n]^\hor)(i,j)$ of $(\Env [n]^\hor)_1$
\end{construction}

\begin{lemma}\label{lem:composites-charQij}
	Let $\DD$ be a Moritable double category and $\sigma \colon \Env [n]^\hor \to \DD$ an object of $\LaxMor(\DD)$. Then $\sigma$ is composite if and only if for any $0 \leq i, j \leq n$, the diagram
	\[(\Delta^{\times j-i-1,\op})^\triangleright \xrightarrow{\overline{Q}_{ij}} (\Env [n]^\hor)(i,j) \xrightarrow{\sigma} \DD(\sigma(i),\sigma(j))\]
	is a colimit diagram.
\end{lemma}

\begin{proof}
	This is essentially the statement that the relative tensor products $M_{ij} \otimes_{A_j} M_{jk}$ are associative, which follows from the fact that $\DD$ is Moritable. We leave the details to the reader.
\end{proof}

This allows us to characterize composite objects in terms of idle slice categories, which will be important for showing that $\Mor(\DD)$ satisfies the Segal condition.

\begin{definition}[The idle slice category]
	The category $\idln$ is defined as the full subcategory of $\alln$ spanned by the idle maps.
\end{definition}

By \cite[Lemma 4.14]{Haugseng2017HigherMoritaCategory}, the source projection $\idln \to \Delta^\op$ is a generalized planar $\infty$-operad and hence $\Env \idln$ is a double subcategory of $\Env \alln = \Env [n]^\hor$.

\begin{lemma}\label{lem:composites-idle-char}
	Let $\DD$ be a Moritable double category. Then the restriction functor
	\[\LaxMor(\DD)_n \simeq \Fun_{/\Delta^\op}^\inert(\alln, \DD) \to \Fun_{/\Delta^\op}^\inert(\idln,\DD)\]
	admits a fully faithful left adjoint whose essential image consists precisely of the composite objects.
\end{lemma}

\begin{proof}
	We will construct this fully faithful left adjoint by vertical left Kan extension along the inclusion $i \colon \idln \hookrightarrow \alln$.
	Conditions \ref{VLan:cond1}--\ref{VLan:cond2} of \cref{lem:VLan-existence} are clearly satisfied.
	For condition \ref{VLan:cond3}, given $0 \leq i < j \leq n$, let $H_{ij}$ denote the unique horizontal morphism from $i$ to $j$ in $[n]^\hor$.
	Observe that the functor $\overline{Q}_{ij}$ from \cref{const:barQij} lifts uniquely to a functor
	\[(\Delta^{\times j-i-1,\op})^\triangleright \to (\Env [n]^\hor)(i,j)_{/H_{ij}},\]
	and that its restriction to $\Delta^{\times j-i-1,\op}$ defines a functor
	\[\Delta^{\times j-i-1,\op} \to (\Env \idln)(i,j)_{/H_{ij}}.\]
	We leave it to the reader to verify this functor is right cofinal, hence that \ref{VLan:cond3} is satisfied.
	This guarantees that the fully faithful left adjoint exists and is given by $\VLan_i$.

	To show that its essential image consists of the composite objects, note that by \cref{lem:primi2} and the above, $\sigma \colon \Env [n]^\hor \to \DD$ lies in the essential image of $\VLan_i$ precisely if
	\[(\Delta^{\times j-i-1,\op})^\triangleright \xrightarrow{\overline{Q}_{ij}} (\Env [n]^\hor)(i,j) \xrightarrow{\sigma} \DD(\sigma(i),\sigma(j))\]
	is a colimit diagram for any $0 \leq i < j \leq n$.
	By \cref{lem:composites-charQij}, this is the case precisely if $\sigma$ is composite.
\end{proof}

Combining this with the results of \cite[\S 4.3]{Haugseng2017HigherMoritaCategory}, it follows that $\Mor(\DD)$ is a double category.

\begin{proposition}\label{prop:MorD-is-a-double-category}
	Let $\DD$ be a Moritable double category. Then $\Mor(\DD)$ is a double category.
\end{proposition}

\begin{proof}
	Observe that the inclusions $\rho_{i,i+j} \colon [j] \to [n]$ induce maps $\all{j} = \idl{j} \hookrightarrow \idln$ for $j=0,1$. It is proved in \cite[\S 4.3]{Haugseng2017HigherMoritaCategory} that these exhibit $\idln$ as the colimit
	\[\all{1} \cup_{\all{0}} \cdots \cup_{\all{0}} \all{1}\]
	in the category of generalized planar $\infty$-operads. It follows from \cref{lem:composites-idle-char} that the Segal maps
	\[\Fun_{/\Delta^\op}^\inert(\idln,\DD) \simeq \Mor(\DD)_n \to \Mor(\DD)_1 \times_{\Mor(\DD)_0} \cdots \times_{\Mor(\DD)_0}\Mor(\DD)_1\]
	are equivalences.
\end{proof}

\begin{remark}\label{remark:Composites-left-adjoint-to-lax-Segal-map}
	This proof shows that the the Segal map
	\[\LaxMor(\DD)_n \to \LaxMor(\DD)_1 \times_{\LaxMor(\DD)_0} \cdots \times_{\LaxMor(\DD)_0}\LaxMor(\DD)_1\]
	admits a fully faithful left adjoint whose essential image consists of the composite objects.
\end{remark}

\begin{observation}\label{obs:characterizing-functors-into-morD}
	Let $\DD$ be a Moritable double category and $\DE$ a double category. Since $\Mor(\DD)$ is levelwise a full subcategory of $\LaxMor(\DD)$, it follows that $\Fun(\DE,\Mor(\DD))$ is a full subcategory of $\Fun(\DE,\LaxMor(\DD))$.
	It follows from the definition of the composite objects in $\LaxMor(\DD)_n$ that $F \colon \DE \to \LaxMor(\DD)$ lands in $\Mor(\DD)$ precisely if the corresponding functor $F' \colon \Env(\DE) \to \DD$ has the following property:
	\begin{itemize}
		\item Let $M \colon x \to y$ and $N \colon y \to z$ be horizontal morphisms in $\DE$ and $K \colon x \to z$ their composite in $\DE$. Then the augmented simplicial diagram
		\[
		\begin{tikzcd}
		    F(K) & F(M) \otimes F(N) \ar[l] \ar[r, shorten=0.4em] & F(M) \otimes F(\unit_y) \otimes F(N) \ar[l, shift left] \ar[l, shift right] \ar[r, shift left, shorten=0.4em] \ar[r, shift right, shorten=0.4em] & \cdots \ar[l, shift left=2] \ar[l] \ar[l, shift right=2]
		\end{tikzcd}
		\]
		is a colimit diagram in $\DD(x,z)$; that is, it exhibits $F(K)$ as the relative tensor product $F(M) \otimes_{F(\unit_y)} F(N)$.
	\end{itemize}
	This augmented simplicial diagram is constructed analogously to the functor $Q_{ijk}$ in \cref{const:Qijk}.
\end{observation}

\section{The proof of the universal property}

Now that we have constructed the double category $\Mor(\DD)$, it is time to prove its universal property. It follows from \cref{proposition:Partially-cofree-fibrations} that $\UnitLax(\DE,\Mor(\DD)) \simeq \Fun(\DE,R_\idle(\un\Mor(\DD)))$. In particular, if we can show that $R_\idle(\un\Mor(\DD)) \simeq \LaxMor(\DD)$, then the result follows from \cref{proposition:Universal-property-LaxMor}. By the explicit description of $R_\idle(\un\Mor(\DD))$ given in \cref{proposition:Partially-cofree-fibrations}, this holds precisely if we can establish an equivalence
\begin{equation}\label{required-natural-equivalence}
	\Fun([n]^\hor,\LaxMor(\DD)) \simeq \UnitLax([n]^\hor,\Mor(\DD))
\end{equation}
that is natural in $[n]$. Observe that by \cref{lem:composites-idle-char}, the inclusions $\Mor(\DD)_n \hookrightarrow \LaxMor(\DD)_n$ admit right adjoints. By the dual of \cite[Proposition 7.3.2.6]{HA} or an argument dual to \cite[Lemma 4.45]{Abellan2023ComparingLaxFunctors}, the inclusion $\cL \colon \un \Mor(\DD) \to \un \LaxMor(\DD)$ admits a relative right adjoint 
\begin{equation}\label{equation:Right-adjoint-Laxmor}
	\cR \colon \un \LaxMor(\DD) \to \un \Mor(\DD)
\end{equation}
that preserves cocartesian lifts of idle maps---in particular, $\cR$ is an object of $\UnitLax(\LaxMor(\DD),\Mor(\DD))$. The natural equivalence \cref{required-natural-equivalence} will be obtained by postcomposition with $\cR$.

Observe that the relative adjunction $\cL \dashv \cR$ induces an adjunction
\begin{equation}\label{equation:important-adjunction-appendix}
	\cL_* : \UnitLax([n]^\hor, \Mor(\DD)) \rightleftarrows \UnitLax([n]^\hor,\LaxMor(\DD)) : \cR_*
\end{equation}
by postcomposition and that $\cL_*$ is fully faithful since $\cL$ is. The proof of \cref{theoremB:Morita} will rely on an explicit characterization of the essential image of $\cL_*$. Note that the right-hand side of \cref{equation:important-adjunction-appendix} may be identified with $\Fun(\Env \UnitEnv [n]^\hor, \DD)$. Let us write $\eta \colon [n]^\hor \to \UnitEnv[n]^\hor$ and $\varepsilon \colon \UnitEnv [n]^\hor \to [n]^\hor$ for the unit and counit, respectively.

\begin{lemma}\label{the crucial lemma}
	Let $\DD$ be a Moritable double category. Then the functor
	\[\eta^* \colon \Fun(\Env \UnitEnv [n]^\hor, \DD) \simeq \Lax(\UnitEnv [n]^\hor, \DD) \to \Lax([n]^\hor, \DD)\]
	obtained by restricting along $\eta \colon [n]^\hor \to \UnitEnv [n]^\hor$ admits a fully faithful left adjoint $\eta_!$. Moreover, the essential image of $\eta_!$ agrees with the essential image of
	\[\cL_* \colon \UnitLax([n]^\hor, \Mor(\DD)) \hookrightarrow \UnitLax([n]^\hor,\LaxMor(\DD)) \simeq \Fun(\Env \UnitEnv [n]^\hor, \DD).\]
\end{lemma}

The proof of \cref{the crucial lemma} will be given in \cref{appendix:ssec:Proof-crucial-lemma} below. Let us first show how to deduce our main theorem from it.

\begin{theorem}[Universal property of $\Mor(\DD)$]\label{theorem:appendix:universal-property-Mor}
	Let $\DD$ be a Moritable double category. Then postcomposition with the functor $\cR$ from \cref{equation:Right-adjoint-Laxmor} induces an equivalence
	\[\Fun([n]^\hor,\LaxMor(\DD)) \simeq \UnitLax([n]^\hor,\Mor(\DD)).\]
	In particular, for any double category $\DE$ there is a natural equivalence
	\[\Lax(\DE,\DD) \simeq \UnitLax(\DE,\Mor(\DD)).\]
\end{theorem}

\begin{proof}
	Consider the diagram
	\begin{equation*}\label{equation:Diagram1-morita-theorem}
		\begin{tikzcd}
			\Fun([n]^\hor,\LaxMor(\DD)) & \UnitLax([n]^\hor,\Mor(\DD)). \\
			\UnitLax([n]^\hor,\LaxMor(\DD))
			\arrow[from=1-1, to=1-2]
			\arrow[from=1-1, to=2-1,hook,"i"]
			\arrow[""{name=0, anchor=center, inner sep=0},"\eL_*", bend left=15, from=1-2, to=2-1]
			\arrow[""{name=1, anchor=center, inner sep=0},"\eR_*", from=2-1, to=1-2]
			\arrow["\dashv"{anchor=center, rotate=108}, draw=none, from=0, to=1]
		\end{tikzcd}
	\end{equation*}
	We need to show that the top horizontal map is an equivalence. By \cref{proposition:Universal-property-LaxMor,the crucial lemma}, this diagram is equivalent to a diagram of the form
	\[\begin{tikzcd}
		\Lax([n]^\hor,\DD) & \Lax([n]^\hor,\DD). \\
		\Lax(\UnitEnv [n]^\hor,\DD)
		\arrow[from=1-1, to=1-2]
		\arrow[from=1-1, to=2-1,hook,"\varepsilon^*"']
		\arrow[""{name=0, anchor=center, inner sep=0},"\eta_!", bend left=15, from=1-2, to=2-1]
		\arrow[""{name=1, anchor=center, inner sep=0},"\eR_*'", from=2-1, to=1-2]
		\arrow["\dashv"{anchor=center, rotate=108}, draw=none, from=0, to=1]
	\end{tikzcd}\]
	Since $\eR'_*$ is right adjoint to $\eta_!$, we see that $\eR'_* \simeq \eta^*$. Since $\varepsilon \eta \simeq \id$ by the triangle equalities, we obtain an equivalence $\eta^* \varepsilon^* \simeq \id$. It follows that the map
	\[\Fun([n]^\hor,\LaxMor(\DD)) \xrightarrow{\eR_* \circ i} \UnitLax([n]^\hor,\Mor(\DD))\]
	is an equivalence.
	This shows that $\LaxMor(\DD) \simeq R_\idle(\un\Mor(\DD))$, so the result follows from \cref{proposition:Partially-cofree-fibrations}.
\end{proof}

\begin{remark}
	In fact, it follows from \cref{proposition:Partially-cofree-fibrations} that there is a natural equivalence
	\[\Lax(\DE,\DD) \simeq \Fun(\DE,\LaxMor(\DD)) \simeq \UnitLax(\DE,\Mor(\DD))\]
	for any functor $\DE \colon \Delta^\op \to \Cat$, not just when $\DE$ is a double category.
\end{remark}

\section{The proof of \texorpdfstring{\cref{the crucial lemma}}{the crucial lemma}}\label{appendix:ssec:Proof-crucial-lemma}

For the proof of \cref{the crucial lemma}, we need to compute vertical left Kan extensions along $\eta \colon [n]^\hor \to \UnitEnv [n]^\hor$. For this it will be useful to introduce some notation.

\begin{notation}
Given $0 \leq i < j \leq n$, we shall write $H_{ij}$ for the horizontal arrow in $[n]^\hor$ with source $i$ and target $j$, and $\unit_i$ for the identity at $i$.
We will denote horizontal composition in $\UnitEnv [n]^\hor$ with $\barotimes$.
\Cref{lemma:Unital-lax-envelope} shows that any non-identity horizontal morphism in $\UnitEnv [n]^\hor$ can be uniquely written as a composite of the form $H_{i_0 i_1} \barotimes \cdots \barotimes H_{i_{m-1}i_m}$ for some $0 \leq i_0 < \cdots < i_m \leq n$.
We will write $\boxtimes$ for the horizontal composites in $\Env [n]^\hor$ and $\Env \UnitEnv [n]^\hor$, and $A_i$ for the image of $\unit_i$ under the universal lax functors $[n]^\hor \to \Env [n]^\hor$ and $\UnitEnv [n]^\hor \to \Env \UnitEnv [n]^\hor$.
\end{notation}

\begin{proof}[Proof of \cref{the crucial lemma}]
	We will apply \cref{lem:VLan-existence} to construct the left adjoint $\VLan_\eta$ of
	\[\eta^* \colon \Lax(\UnitEnv [n]^\hor, \DD) \to \Lax([n]^\hor, \DD).\]
	It is clear that conditions \ref{VLan:cond1}--\ref{VLan:cond2} of \cref{lem:VLan-existence} are satisfied. For condition \ref{VLan:cond3}, let $H = H_{i_0 i_1} \barotimes \cdots \barotimes H_{i_{m-1}i_m}$ be a primitive horizontal arrow in $(\Env \UnitEnv [n]^\hor)(i_0,i_m)$.
	We leave it to the reader to verify that the functor
	\begin{align*}
	(\Delta^\op)^{\times m-1} &\to (\Env [n]^\hor)(i_0,i_m)_{/H} \\
	([k_1],\ldots,[k_{m-1}]) &\mapsto H_{i_0 i_1} \boxtimes A_{i_1}^{\boxtimes k_1} \boxtimes H_{i_1 i_2} \boxtimes A_{i_2}^{\boxtimes k_2} \boxtimes \cdots \boxtimes A_{i_{m-1}}^{\boxtimes k_{m-1}} \boxtimes H_{i_{m-1} i_m}
	\end{align*}
	is right adjoint and hence right cofinal.
	Since the diagonal $\Delta^\op \to (\Delta^\op)^{\times m-1}$ is right cofinal, condition \ref{VLan:cond3} is satisfied.

	By \cref{lem:primi2} and the cofinality of this functor, it follows that $G \colon \Env \UnitEnv [n]^\hor \to \DD$ lies in the image of $\VLan_\eta$ precisely if for any horizontal arrow $H = H_{i_0 i_1} \barotimes \cdots \barotimes H_{i_{m-1}i_m}$ in $\UnitEnv [n]^\hor$, the map
	\[G(H_{i_0i_1}) \otimes_{G(A_{i_1})} \cdots \otimes_{G(A_{i_{m-1}})} G(H_{i_{m-1}i_m}) \to G(H_{i_0 i_1} \barotimes \cdots \barotimes H_{i_{m-1}i_m})\]
	is an equivalence.
	It follows from the associativity of the relative tensor product and \cref{obs:characterizing-functors-into-morD} that this is the case precisely if the corresponding $G' \colon \UnitEnv [n]^\hor \to \LaxMor(\DD)$ lands in $\Mor(\DD)$.
	In particular, the essential image of $\VLan_\eta$ in $\Lax(\UnitEnv [n]^\hor, \DD) \simeq \Fun(\UnitEnv [n]^\hor, \LaxMor(\DD))$ agrees with that of $\eL_* \colon \Fun(\UnitEnv [n]^\hor, \Mor(\DD)) \to \Fun(\UnitEnv [n]^\hor, \LaxMor(\DD))$
\end{proof}